\documentclass[11pt,reqno]{amsart}
\usepackage{graphicx}
\usepackage{color}
\usepackage{amsmath,amssymb,amsthm,amsfonts}
\allowdisplaybreaks

\usepackage{graphicx}
\usepackage{color}
\usepackage{multicol}
\usepackage{cases,indentfirst}
\def\R{\mathbb{R}}
\makeatletter

\newcommand{\Rmnum}[1]{\expandafter\@slowromancap\romannumeral #1@}
\makeatother
\newcommand{\D}{\displaystyle}
\newtheorem{thm}{Theorem}[section]
\newcommand{\norm}[1]{\left\lVert#1\right\rVert}

\newtheorem{lemma}[thm]{Lemma}

\newtheorem{theorem}[thm]{Theorem}
\newtheorem{proposition}[thm]{Proposition}

\usepackage[numbers,sort&compress]{natbib}
\begin{document}

\author{Qingqing Liu}
\address{
School of Mathematics, South China University of Technology,
Guangzhou, 510641, P. R. China} \email{maqqliu@scut.edu.cn}

\author{Hongyun Peng}
\address{School of Mathematics and Statistics, Guangdong University of Technology, Guangzhou, 510006,
China} \email{penghy010@163.com}

\author{Zhi-an Wang*}
\address{Corresponding Author,
Department of Applied Mathematics, Hong Kong Polytechnic University,
Hung Hom, Kowloon, Hong Kong} \email{mawza@polyu.edu.hk}

\title[Asymptotic Stability of a quasi-linear hyperbolic-parabolic model]{Convergence to  nonlinear diffusion waves for a  hyperbolic-parabolic chemotaxis system modelling vasculogenesis}

\begin{abstract}
\normalsize{In this paper, we are concerned with a quasi-linear hyperbolic-parabolic system of persistence and endogenous chemotaxis modelling vasculogenesis in $\mathbb{R}$. Under some suitable structural assumption on the pressure function, we first predict the system admits a nonlinear diffusion wave in $\mathbb{R}$ based on the empirical results in the literature. Then we show that the solution of the concerned system will locally and asymptotically converges to this nonlinear diffusion wave if the wave strength is small. By using the time-weighted energy estimates, we further prove that the convergence rate of the nonlinear diffusion wave is algebraic.}

\vspace*{4mm}

\vspace*{1mm}
 
\end{abstract}

%\subjclass[2020]{35A01, 35B40, 35B45, 35K57, 35Q92, 92C17}

%\keywords{Chemotaxis, hyperbolic-parabolic system, diffusion wave, asymptotic stability}

\maketitle

\numberwithin{equation}{section}
\bigbreak
\section{Introduction}

In order to depict the key characteristics of the {\it in vitro} experiment of blood vessels, showing that the cells randomly scattered on the gel matrix will automatically organize into a network of connected blood vessels,  Ambrosi et al. in \cite{ambrosi2005review, Gamba} proposed the following quasi-linear hyperbolic-parabolic  chemotaxis  system
\begin{eqnarray}\label{HPV1}
	&&\left\{\begin{aligned}
		& \partial_{t}\rho+\nabla\cdot(\rho u)=0,\\
		&\partial_t(\rho u)+\nabla\cdot(\rho u\otimes u)+\nabla p(\rho)=\mu \rho\nabla \phi-\alpha \rho u,\\
		&\partial_t \phi=D\Delta \phi+a \rho-b \phi.
	\end{aligned}\right.
\end{eqnarray}
Here the unknowns $\rho=\rho(x,t)\geq 0 $ and $
u=u(x,t)\in \mathbb{R}^{n} (n\geq 1)$ denote the density and velocity of endothelial  cells, respectively, and  $\phi=\phi(x,t)\geq 0$ denotes the concentration of the chemoattractant secreted by the endothelial  cells. The convection term $\nabla \cdot (\rho u \otimes u)$ models the cell movement persistence (inertial effect), $p(\rho)$ is the cell-density dependent pressure function accounting for the fact that closely packed cells resist to compression due to the impenetrability of cellular matter, the parameter $\mu>0$ measures the intensity of cell response to the chemoattractant concentration gradient and $-\alpha \rho u$ corresponds to a damping (friction) force with coefficient $\alpha>0$ as a result of the interaction between cells and the underlying substratum; $D>0$ is the diffusivity of the chemoattractant, the positive constants $a$ and $b$ denote the secretion and death rates of the chemoattractant, respectively.

The hyperbolic or hyperbolic-parabolic chemotaxis models with different structures from \eqref{HPV1} have been studied in the literature (cf. \cite{HLWW,LW2009,Perthame-book, MWZ}). Nevertheless the mathematical structure of \eqref{HPV1} is analogous to the well-known damped Euler-Poisson system where $\phi$ is called the flow potential satisfying a Poisson (elliptic) equation: $-\Delta \phi= \rho$, which has numerous essential applications such as propagation of electrons in semiconductor devices (cf. \cite{markowich2012semiconductor}) and the transport of ions in plasma physics (cf. \cite{choudhuri1998physics}) when $\mu<0$ and $\alpha\geq0$, as well as the collapse of gaseous stars due to self-gravitation \cite{chandrasekhar1957introduction} when $\mu>0$ and $\alpha=0$. Here the sign of $\mu$ corresponds to attractive or repulsive forces similar to the attractive or repulsive chemotaxis (cf. \cite{Jin-W}). Due to the essential difference of analysis between elliptic and parabolic equations, the analytical tools developed for the damped Euler-Poisson system is not directly applicable to the system \eqref{HPV1}. Up to date, there are not many analytical results available to the system \eqref{HPV1}.
%\begin{eqnarray}\label{1.2}
%[\rho,u,\phi]|_{t=0}=[\rho_0,u_{0},\phi_{0}](x)\rightarrow [\bar{\rho}, 0,\bar{\phi}]\quad \mathrm{as} \quad |x|\rightarrow \infty
%\end{eqnarray}
%for some  constants  $\bar{\rho}>0$ and $\bar{\phi}>0$.
When the initial value $({\rho}_{0}(x), {u}_{0}(x), {\phi}_0(x)) \in H^s(\R^n) (s>n/2+1)$  is a small perturbation of the constant ground state (i.e. equilibrium) $({\rho}_c, 0, {\phi}_c)$ with  ${\rho}_{c}>0$ sufficiently small, it was shown in \cite{Russo12, Russo13} that the system \eqref{HPV1} admits global strong solutions without vacuum  converging to $({\rho}_c, 0, {\phi}_c)$ in $L^2$-norm  with an algebraic decay rate. As $\alpha\rightarrow \infty$ (strong damping), it was formally derived in \cite{chavanis2007kinetic} by the asymptotic analysis and subsequently justified in \cite{FD} that the solution of \eqref{HPV1} converges to that of a parabolic-elliptic Keller-Segel type chemotaxis system. The asymptotic behavior of solutions to \eqref{HPV1} and its limiting Keller-Segel system was further compared numerically in \cite{NRT2}. By adding a viscous term $\Delta u $ to the second equation of \eqref{HPV1}, the linear stability of the constant ground state $({\rho}_c, 0, {\phi}_c)$ was obtained in \cite{kowalczyk2004stability} under the condition
\begin{equation}\label{pressure}
bp'({\rho}_c)- a \mu {\rho}_c>0.
\end{equation}
A typical form of $p$ fulfilling \eqref{pressure} is $ p(\rho) = \frac{K}{2}\rho ^{2}$ with $ K > \frac{a \mu}{b}$.
The stationary solutions of one dimensional \eqref{HPV1} with vacuum (bump solutions) in a bounded interval with zero-flux boundary condition were constructed in \cite{berthelin2016stationary, carrillo2019phase, NRT1}. The model \eqref{HPV1} with $p(\rho)=\rho$ and periodic boundary conditions in one dimension was numerically explored in \cite{filbet2005approximation}. Recently the stability of transition layer stationary solutions of \eqref{HPV1} on $\R_+=[0,\infty)$ was established in \cite{HPWZ-JLMS}.
%It is noted that in all afore-mentioned results except \cite{HPWZ-JLMS}, the large-time limit of solutions to the time-dependent problem \eqref{HPV1} is constant equilibria when the initial value is imposed as a small perturbation of these constant equilibria. However, the networks patterns observed in the experiment of \cite{Gamba} are essentially non-constant and hence exploration of non-constant solutions will more relevant.   The results of \cite{berthelin2016stationary, carrillo2019phase, NRT1, HPWZ-JLMS} have shown that the system \eqref{HPV1} has bump or transition-layer solutions.

It is well known that damping is a factor triggering diffusion waves in many hyperbolic system such as the $p$-system or damped Euler equations (cf. \cite{HL,JZ1,JZ2,Mei1,Mei2, MM,N1,N2,Nishihara2,NY,ZJ} without vacuum and \cite{DP,HMP, MM-JDE, Pa} with vacuum), as well as the  Euler-Poisson system of semiconductors \cite{GHL}, bipolar Euler-Maxwell equation \cite{DLZ1}, Timoshenko system \cite{IK} and the radiating gas model \cite{LK}.
Motivated by the structural analogue between the Euler-Poisson system and \eqref{HPV1} with dampings, the authors have shown recently in \cite{LPW} that the hyperbolic-parabolic system \eqref{HPV1} admits linear diffusion waves in $\mathbb{R}^3$ which are locally asymptotically stable by the Fourier and spectral analysis. Under the framework of \cite{LPW}, the time decay of solutions decreases with respect to the space dimension and the energy estimates will lose time integrability in $\mathbb{R}^n (n=1,2)$. Therefore whether the hyperbolic-parabolic chemotaxis system \eqref{HPV1} admits stable diffusion waves in $\R$ or $\mathbb{R}^2$ remains unknown. The purpose of this paper is to show that the system \eqref{HPV1} in $\R$ admits nonlinear diffusion waves which are stable against a small perturbation by using the technique of taking antiderivative, unlike the framework of \cite{LPW}.

To demonstrate our ideas, we  set $m=\rho u$ for convenience, namely $m$ denotes the momentum of cells, and recast the system (\ref{HPV1}) for $(x,t)\in \R\times \R^+$  as
\begin{equation}\label{va1d}
\left\{\begin{array}{ll}
			 \rho_t +m_x =0, & (x,t)\in \R\times \R^+,\\[1mm]
			m_t +\big(\frac{m^2}{\rho}+p(\rho) \big)_x= \mu \rho\phi_x-\alpha m, & (x,t)\in \R\times \R^+,\\[1mm]
			\phi_t=D\phi_{xx}+a\rho-b\phi, & (x,t)\in \R\times \R^+,
\end{array}\right.
\end{equation}
where we prescribe the initial data
\begin{eqnarray}\label{initial}
	U|_{t=0}=U_{0}:=(\rho_{0},m_{0},\phi_{0})\longrightarrow
	(\rho_{\pm},0,\phi_{\pm}),\ \ \ \ \ \ \ \mathrm{as}\ \ \ x\rightarrow
	\pm\infty,
\end{eqnarray}
with $\rho_{\pm}>0, \ \ \rho_{+}\neq \rho_{-}$ and $ \phi_{\pm}=\frac{a}{b}\rho_{\pm}$, where we assume $\rho_{\pm}>0$ to avoid possible vacuum. Moreover we impose the following conditions on the pressure function
\begin{equation}\label{assumpp}
	p(\rho)\in C^{3}(\mathbb{R}^{+}),\ \ \ p'(\rho)-\frac{a\mu}{b}\rho>0, \ \ \mathrm{for\ any}\ \rho >0.
\end{equation}
Due to the external damping (frictional) force, one may expect that the inertial term in the momentum equation of \eqref{va1d} decays to zero faster than other terms so that the pressure gradient force is balanced by the frictional force plus the potential force.  Hence we may predict that the solution $(\rho, m, \phi)$ of  \eqref{va1d}-\eqref{initial} will behave time asymptotically as the solution $(\bar{\rho}, \bar{m}, \bar{\phi})$ to the following equations
\begin{equation}\label{va1d1}
	\begin{cases}
		\begin{split}
			& \bar{\rho}_t +\bar{m}_x =0, \\
			&(p(\bar{\rho}) )_x= \mu \bar{\rho}\bar{\phi}_x-\alpha\bar{m},\\
			&a\bar{\rho}=b\bar{\phi},
		\end{split}
	\end{cases}
\end{equation}
or equivalently, by denoting $q(\bar{\rho})=p(\bar{\rho})-\frac{a\mu}{2b}\bar{\rho}^2$,
\begin{subequations}\label{va1d1rh}
\begin{numcases}
	\/\bar{\rho}_t-\Big(\frac{1}{\alpha}q(\bar{\rho})\Big)_{xx}=0, \label{va1d1rh1}\\
			\bar{m}=-\frac{1}{\alpha}\left[p(\bar{\rho})-\frac{a\mu}{2b}\bar{\rho}^2\right]_{x}, \label{va1d1rh2}\\
			\bar{\phi}=\frac{a}{b}\bar{\rho}, \label{va1d1rh3}
\end{numcases}
\end{subequations}
with the following asymptotic states at far fields
\begin{equation}\label{barrho}
	(\bar{\rho}(x,t),\bar{m}(x,t),\bar{\phi}(x,t)) \longrightarrow \left (\rho_{\pm},0,\frac{a}{b}\rho_{\pm}\right), \ {\rm as} \ x\rightarrow \pm \infty.
\end{equation}
In our paper, without loss of generality, we assume $\rho_{-}<\rho_{+}$. As shown in \cite{Duyn}, the equation  $\eqref{va1d1rh1}$ admits  unique nonlinear diffusion wave
$\bar{\rho}(x,t)=\varphi(\xi)$ with $\xi=\frac{x}{\sqrt{1+t}}$ under the condition \eqref{assumpp} such that $\varphi(\xi)=\rho_{\pm}$. Then substituting this into  $\eqref{va1d1rh2}$ and $\eqref{va1d1rh3}$, we find a unique solution $(\bar{\rho}, \bar{m},\bar{\phi})(x,t)$ of diffusion wave for $\eqref{va1d1rh}$ with
\begin{equation}\label{dw}
\bar{\rho}(x,t)=\varphi\bigg(\frac{x}{\sqrt{1+t}}\bigg), \ \bar{m}(x,t)=-\frac{1}{\alpha}\left[p(\bar{\rho})-\frac{a\mu}{2b}\bar{\rho}^2\right]_{x}, \ \bar{\phi}(x,t)=\frac{a}{b}\bar{\rho}.
\end{equation}

The aim of this paper is to show if the initial value $(\rho_0, m_0, \phi_0)$ satisfying \eqref{initial} is a small perturbation of $(\bar{\rho}(x,0), \bar{m}(x,0), \bar{\phi}(x,0)$,  then the system \eqref{va1d} with \eqref{assumpp} admits a unique solution whose asymptotic profile is the nonlinear diffusion wave $(\bar{\rho}, \bar{m},\bar{\phi})(x,t)$ given by \eqref{dw}. To state our main results, we first define the perturbation of $(\bar{\rho}(x,0), \bar{m}(x,0), \bar{\phi}(x,0)$ as follows
	\begin{equation}\label{pert}
	\left\{	\begin{array}{lll}
				V_{0}(x)={\normalsize \int_{-\infty}^x}(\rho_{0}(y)-\bar{\rho}(y+x_0,0))dy, \\[1mm]
				M_{0}(x)=m_{0}(x)-\bar{m}(x+x_0,0), \\[1mm]
				\Phi_{0}(x)=\phi_{0}(x)-\bar{\phi}(x+x_0,0),
\end{array}\right.
	\end{equation}	
	where $x_0$ is a constant uniquely determined (see section 2.2)  such 	that the initial perturbation from the spatially shifted diffusion waves with shift $x_0$ is of integral zero, namely $V_0(\infty)=0$. Above we define $V_{0}(x)$ in a form of anti-derivative of the perturbation because the first equation of \eqref{va1d} is a conservation law for which the technique of taking anti-derivative is usually invoked (cf. \cite{smoller}). By the method of weighted energy estimates, we shall prove the following results in this paper.
\begin{theorem}\label{1-1}
	Let \eqref{assumpp} hold. Then there exists a constant $\epsilon>0$, such that if $(V_{0},M_{0},\Phi_{0})\in H^3(\R)\times H^2(\R)\times H^4(\R)$ satisfies
	\begin{equation*}
		\begin{split}
		\norm{V_{0}}^2_{H^3(\R)}+\norm{M_{0}}^2_{H^2(\R)}+\norm{\Phi_{0}}^2_{H^4(\R)}+|\rho_{+}-\rho_{-}|\leq \epsilon^2,
		\end{split}
	\end{equation*}
	where $(V_{0},M_{0},\Phi_{0})$ is defined in \eqref{pert}, the system \eqref{va1d}-\eqref{initial} possesses a unique global classical solution
	$(\rho, m, \phi)(x, t)$ which converges to the shifted diffusion wave $(\bar{\rho}, \bar{m}, \bar{\phi})(x+x_0,t)$ solving \eqref{va1d1rh} and \eqref{barrho} in $L^\infty(\R)$ with algebraic decay rates:
	\begin{equation}\label{infinitydecay}
		\begin{split}
			\|\partial_{x}^{k}(\rho-\bar{\rho})(t)\|_{L^{\infty}(\R)}&\le C \epsilon(1+t)^{-3 / 4-k/2}, \  k=0,1,\\
			\|\partial_{x}^{k}(m-\bar{m})(t)\|_{L^{\infty}(\R)}&\le C\epsilon(1+t)^{-5 / 4-k/2},\  k=0,1, \\
			\|\partial_{x}^{k}(\phi-\bar{\phi})(t)\|_{L^{\infty}(\R)}&\le C\epsilon(1+t)^{-3 / 4-k/2},\  k=0,1,
		\end{split}
	\end{equation}
where $C>0$ is a constant independent of $t$.
\end{theorem}

The rest of the paper is organized as follows. In Section \ref{sec2},  we present some known results on the diffusion wave solution $\bar{\rho}(x,t)$ of \eqref{va1d1rh} with far field states \eqref{barrho}, and reformulate the original equation \eqref{va1d} against a suitable perturbation. In Section \ref{sec3}, we derive the uniform {\it a-priori} estimates and hence establish the existence of global solutions of reformulated problem. In Section \ref{sec4}, we show the algebraic time asymptotic rate of solutions convergent to the nonlinear diffusion wave and prove the main Theorem \ref{1-1}.

\section{Reformulation of the problem}\label{sec2}
In this section, we shall prove the global existence of solutions to \eqref{va1d1}. We first introduce some notations frequently throughout the paper.

\medskip
\noindent{\it Notations.}  In the sequel, $C$ denotes a generic positive  constant where $C$ may vary in the context. For two
quantities $a$ and $b$, $a\sim b$ means $\lambda a \leq  b \leq
\frac{1}{\lambda} a $ for some constant $0<\lambda<1$. For any
integer $m\geq 0$, we use $H^{m}$ to denote the usual
Sobolev space $H^{m}(\mathbb{R})$. For simplicity, the norm of $ H^{m}$ is denoted by
$\|\cdot\|_{m} $ with $\|\cdot \|=\|\cdot\|_{0}$, and we set
\begin{equation*}
	\|[A,B]\|_{X}=\|A\|_{X}+\|B\|_{X}.
\end{equation*}
Without confusion, we shall abbreviate $\|\cdot\|_{L^p(\R)}$ as $\|\cdot\|_{L^p}$ for $1\leq p\leq \infty$ in the sequel.
\subsection{Nonlinear diffusion waves} In this subsection, we give some properties of nonlinear diffusion wave profile \eqref{dw} that will be used in the paper.
%In this subsection, we will list some known results concerning the self-similar solution of the nonlinear parabolic equation  $\eqref{va1d1rh}_{1}$, \eqref{barrho}  under the assumption \eqref{assumpp}.
Let
\begin{equation*}
	\bar{\rho}(x,t)=\varphi \left(\frac{x}{\sqrt{1+t}}\right)=\varphi(\xi),\ \ \ \ \ -\infty<\xi<\infty.
\end{equation*}
Substituting the above self-similar structure form into  $\eqref{va1d1rh1}$, we find that  $\varphi(\xi)$ satisfies
\begin{equation*}
	\begin{cases}
		\begin{split}
			&\varphi''(\xi)+\frac{\frac{\alpha}{2}\xi+q''(\varphi(\xi))\varphi'(\xi)}{q'(\varphi(\xi))}\varphi'(\xi)=0,\\[2mm]
			&\varphi(\pm\infty)=\rho_{\pm},
		\end{split}
	\end{cases}
\end{equation*}
where $q'(\bar{\rho})=p'(\bar{\rho})-\frac{a\mu}{b}\bar{\rho}>0$.  The existence of unique solution of the above equations has been shown in \cite{HL}. For any $\xi_{0}$, it further follows that
\begin{equation*}
	\varphi'(\xi)=\frac{ \varphi'(\xi_{0})q'( \varphi(\xi_{0}))}{q'( \varphi(\xi))}\exp\left(-\int_{\xi_{0}}^{\xi}\frac{\alpha \eta}{2q'( \varphi(\eta))} d\eta\right).
\end{equation*}
The solution $\varphi(\xi)$ is increasing if $\rho_{-}<\rho_{+}$ and decreasing if $\rho_{-}>\rho_{+}$, and satisfies
\begin{equation*}
	\sum_{k=1}^{6}\left|\frac{d^{k}}{d\xi^{k}}\varphi(\xi)\right|+|\varphi(\xi)-\rho_{+}|_{\xi>0}+|\varphi(\xi)-\rho_{-}|_{\xi<0}\leq C|\rho_{+}-\rho_{-}|\exp(-c\alpha \xi^2),
\end{equation*}
where $c$ is a positive constant independent of $x$ and $t$.  Moreover, we can obtain following $L^p$-estimates of the derivatives of $\bar{\rho}$ (cf. \cite{HMP, HP1, ZCJ, ZJ}).
\begin{lemma}\label{prodecay}
	Let $\bar{\rho}(x,t)$ be the self-similar  solution of $\eqref{va1d1rh1}$ and \eqref{barrho}. Then for $p\in [2,+\infty]$ and $1\leq l+k\leq 6$, we have
	\begin{equation*}
		\norm{\partial_{t}^{l}\partial_{x}^{k}\bar{\rho}(\cdot,t)}_{L^{p}}\leq C |\rho_{+}-\rho_{-}|(1+t)^{-\frac{k}{2}-l+\frac{1}{2p}}.
	\end{equation*}
\end{lemma}

\subsection{ Reformulation of the problem}
Inspired by the work  \cite{HL},  we set the perturbation function around the diffusion wave  $(\bar{\rho},\bar{m},\bar{\phi})$ as
\begin{equation}\label{221}
	\begin{cases}
		\begin{split}
			&V(x,t)=\textstyle \int_{-\infty}^x(\rho(y,t)-\bar{\rho}(y+x_0,t))dy,\\
			&M(x,t)=m(x,t)-\bar{m}(x+x_0,t),\\
			&\Phi(x,t)=\phi(x,t)-\bar{\phi}(x+x_0,t),
		\end{split}
	\end{cases}
\end{equation}
where $x_0$ is a constant uniquely determined by
\begin{equation*}
	\begin{split}
		\int_{-\infty}^{+\infty}(\rho(x,0)-\bar{\rho}(x+x_0,0))dx=0,
	\end{split}
\end{equation*}
namely,
$$
x_{0}=\frac{1}{\rho_{+}-\rho_{-}}\int_{-\infty}^{+\infty}\left(\rho_{0}(y)-\bar{\rho}(y,0)\right) dy.
$$
Define the initial perturbation function as
\begin{equation*}
	\begin{cases}
		&V_{0}(x)=V(x,0)=\int_{-\infty}^x(\rho_{0}(y)-\bar{\rho}(y+x_0,0))dy, \\[1mm]
		&M_{0}(x)=M(x,0)=m_{0}(x)-\bar{m}(x+x_0,0), \\[1mm]
		&\Phi_{0}(x)=\Phi(x,0)=\phi_{0}(x)-\bar{\phi}(x+x_0,0).
	\end{cases}
\end{equation*}
Then upon the  substitution of \eqref{221}, we reformulate our problem \eqref{va1d}-\eqref{initial} as
\begin{equation}\label{va1d1d}
	\begin{cases}
		\begin{split}
			&V_t+M=0, \\
			&M_t+\left(\frac{(M+\bar{m})^2}{V_x+\bar{\rho}}\right)_x+\left[p(V_x+\bar{\rho})-p(\bar{\rho})\right]_x = \mu V_x\Phi_x+\mu V_x\bar{\phi}_x+\mu\bar{\rho}\Phi_x-\alpha M-\bar{m}_t,\\[2mm]
			&\Phi_t=D\Phi_{xx}+aV_x-b\Phi-\bar{\phi}_t+D\bar{\phi}_{xx},
		\end{split}
	\end{cases}
\end{equation}
with initial data $(V_{0},M_{0},\Phi_{0})$ satisfying
\begin{equation}\label{va1d1di}
	(V(x,0),M(x,0),\Phi(x,0))=(V_{0}(x),M_{0}(x),\Phi_{0}(x)) \rightarrow 0, \ \ \ \ \mathrm{as}\ \ x\rightarrow \pm\infty.
\end{equation}
Rewrite \eqref{va1d1d}-\eqref{va1d1di} as
\begin{equation}\label{va1d1f}
	\begin{cases}
		\begin{split}
			&V_{tt}-(p'(\bar{\rho})V_x)_x+\alpha V_t+\mu V_x\Phi_x+\mu V_x\bar{\phi}_x+\mu\bar{\rho}\Phi_x+h_{x}+f_{x}=0,\\[2mm]
			&\Phi_t=D\Phi_{xx}+aV_x-b\Phi+g,
		\end{split}
	\end{cases}
\end{equation}
with initial data
\begin{equation}\label{va1d1fi}
	(V(x,0),V_{t}(x,0), \Phi(x,0))=(V_{0}(x),-M_{0}(x), \Phi_0(x)) \rightarrow 0, \ \ \ \ \mathrm{as}\ \ x\rightarrow \pm\infty,
\end{equation}
where
\begin{equation}\label{Defhf}
	\begin{split}\begin{cases}
			h=\D-\frac{(V_t+\frac{1}{\alpha}q(\bar{\rho})_x)^2}{V_x+\bar{\rho}},\\[2mm]
			f=\frac{1}{\alpha}(q(\bar{\rho}))_{t}-\left[p(V_x+\bar{\rho})-p(\bar{\rho})-p'(\bar{\rho})V_x\right],\\[2mm]
			g=-\bar{\phi}_t+D\bar{\phi}_{xx}.
\end{cases}\end{split}\end{equation}

For system \eqref{va1d1d}-\eqref{va1d1di}, we shall establish the following results.
\begin{proposition}\label{1.1}
Let $\delta_{0}=|\rho_{+}-\rho_{-}|$ and \eqref{assumpp} hold.  If $(V_{0},M_{0},\Phi_{0})\in H^3\times H^2\times H^3$ and
	\begin{equation*}
		\begin{split}
\norm{[V_{0},\Phi_{0}]}^2_{3}+\norm{M_{0}}^2_{2}+\delta_{0}\leq \varepsilon_{0}^2,
		\end{split}
	\end{equation*}
	for some sufficiently small constant $\varepsilon_0>0$, then problem  \eqref{va1d1d}-\eqref{va1d1di} admits a unique global classical solutions $(V,M,\Phi)$ satisfying
	\begin{equation*}
			(V(x, t), M(x,t), \Phi(x, t)) \in L^{\infty}\left([0, \infty) ; H^{3}\times H^2\times H^3\right),
%\ M(x, t) \in L^{\infty}\left([0, \infty) ; H^{2}\right)
	\end{equation*}
and
\begin{equation*}
		\sup_{t\geq 0}\norm{V(t)}_{3}^2+\norm{V_{t}(t)}_2^2+\norm{\Phi(t)}_{3}^2
		\le C\left(\norm{V_0}^2_3+\norm{M_0}^2_2+\norm{\Phi_0}^2_3+\delta_{0}\right)\leq C\varepsilon_0^2.
\end{equation*}
Moreover, if there exists a constant $\epsilon$ such that
 	\begin{equation*}
 	\begin{split}
 		\delta_{0}+\norm{V_{0}}^2_{3}+\norm{M_{0}}^2_{2}+\norm{\Phi_{0}}_{4}^2\leq \epsilon^2,
 	\end{split}
 \end{equation*}
then the solution  $(V,M,\Phi)$ has the following decay
	\begin{equation}\label{decayn}
		\begin{split}
			\|\partial_{x}^{k} V_{x}(t)\|_{L^{2}}&\le C \epsilon(1+t)^{-(k+1) / 2}, \ k=0,1,2,\\[1mm]
			\|\partial_{x}^{k} M(t)\|_{L^{2}}&\le C \epsilon(1+t)^{-(k+2) / 2}, \ k=0,1,2,\\[1mm]
			\|\partial_{x}^{k} \Phi(t)\|_{L^{2}}&\le C \epsilon(1+t)^{-(k+1) / 2}, \ k=0,1,2.
		\end{split}
	\end{equation}

where $C>0$ is a positive constant independent of $t$.
\end{proposition}
In view of \eqref{221}, Theorem \ref{1-1} is a direct consequence of Proposition \ref{1.1}. Hence we will focus on the proof of Proposition \ref{1.1}. Before proceeding,  we briefly outline the ideas of proving Proposition \ref{1.1} where part of them are inspired from works \cite{N1,Nishihara2}.
%We will obtain Proposition \ref{1.1} in two steps.
First we establish the global existence of smooth solutions to \eqref{va1d1d}-\eqref{va1d1di}, and
then we derive time decay rates of the solution toward diffusion waves. Although such procedures are routine, the desired results are not easy to be achieved due to the coupling of $\Phi$ and $V$. Since we can not expect the exponential decay of $\Phi$ like the electronic field $E$ as in \cite{GHL} or do not want to
impose smallness assumption on  the constant equilibrium as in \cite{Russo13} either,  some new ideas need to be developed in order to control the linear terms $\mu\bar{\rho}\Phi_x$ in the first equation of  \eqref{va1d1f} and  $aV_{x}$ in the second equation of \eqref{va1d1f}. To  this end,  we take up the assumption \eqref{assumpp} which implies that the following matrix
\begin{equation*}
	\left(
	\begin{array}{cc}
		p'(\bar{\rho}) & -\mu\bar{\rho} \\[3mm]
		-\mu\bar{\rho}  & \frac{b\mu \bar{\rho}}{a}
	\end{array}
	\right)=:A(\bar{\rho}),
\end{equation*}
is positive definite. Then for any $(x_{1}, x_{2})$, we have
$$
\min\{\lambda_{1}(\bar{\rho}),\lambda_{2}(\bar{\rho})\}(x_{1}^2+x_{2}^2)\leq p'(\bar{\rho})x_{1}^2-2\mu\bar{\rho}x_{1}x_{2}+\frac{b\mu \bar{\rho}}{a}x_{2}^2\leq \max\{\lambda_{1}(\bar{\rho}),\lambda_{2}(\bar{\rho})\}(x_{1}^2+x_{2}^2),
$$
where $\lambda_{1}(\bar{\rho})>0$ and $\lambda_{2}(\bar{\rho})>0$ are the eigenvalues of  $A(\bar{\rho})$. Since  $\rho_{-}<\bar{\rho}<\rho_{+}$,  there exist two constants $C_{1}>0$ and $C_{2}>0$ such that
\begin{equation}\label{repeatu}
	C_{1}(x_{1}^2+x_{2}^2)\leq p'(\bar{\rho})x_{1}^2-2\mu\bar{\rho}x_{1}x_{2}+\frac{b\mu \bar{\rho}}{a}x_{2}^2\leq C_{2}(x_{1}^2+x_{2}^2).
\end{equation}
Under the condition \eqref{repeatu}, the two linear terms $\mu\bar{\rho}\Phi_x$ and $aV_{x}$ can be absorbed or eliminated in the energy estimates. Due to the coupling effect, we  find that the decay properties of $\Phi$ and $V_{x}$  can be obtained synchronously.  Moreover the term $g$ in the second equation of \eqref{va1d1f} will complicated the weighted energy estimates in the sequel.

 For later use, we recall a Sobolev inequality about the $L^p$ estimate on products of any two or several terms with the sum of the order of their derivatives equal to a given integer (cf. \cite{DRZ}).
\begin{lemma}\label{lemma3.1}
	Let  $\alpha^{1}=\left(\alpha_{1}^{1}, \cdots, \alpha_{n}^{1}\right)$ and  $\alpha^{2}=\left(\alpha_{1}^{2}, \cdots, \alpha_{n}^{2}\right)$  be two multi-indices with
	$\Big|\alpha^{1}\Big|=k_{1},\left|\alpha^{2}\right|=k_{2}$  and set $k=k_{1}+k_{2}$. Then, for $1 \leq p, q, r \leq \infty$ with  $1 / p=1 / q+1 / r$, we have
\begin{eqnarray}\label{sobin}
		\big\|\partial^{\alpha^{1}} u_{1} \partial^{\alpha^{2}} u_{2}\big\|_{L^{p}{(\mathbb{R}^{n})}} \leq C\Big(\big\|u_{1}\big\|_{L^{q}(\mathbb{R}^{n})}\big\|\nabla^{k} u_{2}\big\|_{L^{r}(\mathbb{R}^{n})}+\big\|u_{2}\big\|_{L^{q}(\mathbb{R}^{n})}\big\|\nabla^{k} u_{1}\big\|_{L^{r}(\mathbb{R}^{n})}\Big),
	\end{eqnarray}
	where $C$ is a positive constant.
\end{lemma}

\section{Global existence. Proof of Proposition \ref{1.1}}\label{sec3}
The existence of local-in-time solutions to \eqref{va1d1f}-\eqref{va1d1fi} can be readily established by the standard iteration argument and hence will be stated without proof details.
\begin{proposition}[Local existence]\label{local}
Let the conditions of Proposition \ref{1.1} hold.  Then there exists a positive constant $T_{0}$ depending on $\varepsilon_{0}$ such that the  problem \eqref{va1d1d}-\eqref{va1d1di} admits a unique solution $(V(x, t), M(x,t), \Phi(x, t)) \in L^{\infty}\left([0, T_0) ; H^{3}\times H^2\times H^3\right)$
satisfying
$$
\sup _{t \in\left[0, T_{0}\right]}\left(\|V\|_{3}^{2}+\|M\|_{2}^{2}+\|\Phi\|_{3}^{2}\right) \leq 2 \varepsilon_{0}^2.
$$
\end{proposition}
To extend local solutions to be global in time, it suffices to derive the uniform {\it a priori} estimates of solutions to \eqref{va1d1f}-\eqref{va1d1fi} where $V_t=-M$.
For any given $T>0$, we denote the solution space for the Cauchy problem \eqref{va1d1f}-\eqref{va1d1fi} by
\begin{equation*}
X(T)=\left\{(V,V_{t},\Phi)\in H^3\times H^2\times H^3, \ 0\leq t\leq T\right\}.
\end{equation*}
Assume that the following  {\it a priori} assumption holds:
\begin{equation}\label{priori}
N(T)=\sup_{0<t<T}\left\{\|[V,\Phi](\cdot,t)\|_{3}+\|V_{t}(\cdot,t)\|_{2}\right\}\ll 1.
\end{equation}
By the Sobolev inequality and $\rho_{-}<\bar{\rho}<\rho_{+}$, it follows from \eqref{priori} that
\begin{equation}\label{bound1}
\frac{1}{2}\rho_{-}\leq V_{x}+\bar{\rho}\leq \frac{3}{2}\rho_{+}.
\end{equation}

The we prove the following uniform {\it a priori} estimate.
\begin{proposition}[{\it A priori} estimate]\label{mainpro}
Let the conditions of Proposition \ref{1.1} hold and $(V,V_{t},\Phi)\in
X(T)$ be a smooth solution of \eqref{va1d1f}-\eqref{va1d1fi}. Then there exists a constant $C>0$ independent of $t$ such that
\begin{equation}\label{f1}
\begin{split}
&\norm{V}_{3}^2+\norm{V_{t}}_2^2+\norm{\Phi}_{3}^2+\int_0^T\left(\norm{V_{x}}_2^2+\norm{\Phi}_{3}^2+\norm{\Phi_{t}}_2^2+\norm{V_{t}}_2^2\right)dt\\[2mm]
\le&C\left(\norm{V_0}^2_3+\norm{M_0}^2_2+\norm{\Phi_0}^2_3+\delta_{0}\right)\leq C\varepsilon_0^2.
\end{split}
\end{equation}
\end{proposition}

\begin{proof} We divide the proof into three steps.

\noindent \textbf{ Step 1.} We claim the following inequality hold
\begin{equation}\label{Step1C}
	\begin{aligned}
		&\frac{d}{dt}\sum_{0\leq k\leq 2}\left\{\frac{\alpha}{2}\|\partial^k_{x}V\|^2+\int_{\R}\partial^k_{x}V_{t}\partial^k_{x}Vdx+\frac{\mu}{2a}\int_{\R}\bar{\rho}\left(\partial^k_{x}\Phi\right)^2dx\right\}+\frac{\mu D}{a}\sum_{0\leq k\leq 2}\int_{\R} \bar{\rho}\left(\partial^k_{x}\Phi_{x}\right)^2dx\\
		&+\sum_{0\leq k\leq 2}\left\{\int_{\R} p'(\bar{\rho})\left(\partial^k_{x}V_{x}\right)^2dx-2\mu \int_{\R}\bar{\rho}\partial^k_{x}\Phi\partial^k_{x}V_{x} dx+\frac{b\mu}{a}  \int_{\R}\bar{\rho}\left(\partial^k_{x}\Phi\right)^2dx\right\}\\[3mm]
	\leq & -\frac{\mu}{b}\frac{d}{dt}\int_{\R}\bar{\rho}_{x}\Phi Vdx+\|V_{t}\|_{2}^2+C(N(T)+\delta_{0})(\|V_{x}\|_{2}^2+\|V_{t}\|_{2}^2+\|\Phi\|_{3}^2)+C\delta_{0}(1+t)^{-\frac{5}{4}}.
\end{aligned}
\end{equation}
To this end, we rewrite \eqref{va1d1f} as
\begin{equation}\label{va1d1fxk}
\begin{cases}
\begin{split}
&V_{tt}-(p'(\bar{\rho})V_x)_{x}+\alpha V_{t}+\mu (\bar{\rho }\Phi)_{x}=-\mu V_{x}\Phi_{x}-\mu V_{x}\bar{\phi}_{x}+\mu \bar{\rho}_{x}\Phi -h_{x}-f_x,\\[2mm]
&\Phi_{t}-D\Phi_{xx}+b \Phi-aV_{x}=g.
\end{split}
\end{cases}
\end{equation}
Applying $\partial^{k}_{x}$ for $0\leq k\leq 2$ to the first equation of $\eqref{va1d1fxk}$ and multiplying the result by $\partial^k_{x}V$, one has
\begin{equation}\label{va1d-2}
\begin{split}
& \partial^{k}_{x} V_{tt}\partial^k_{x}V-\partial^{k}_{x}(p'(\bar{\rho})V_x)_{x}\partial^k_{x}V+\alpha \partial^{k}_{x}V_{t}\partial^k_{x}V+\mu \partial^{k}_{x}(\bar{\rho }\Phi)_{x}\partial^k_{x}V\\
=& -\mu \partial^k_{x}(V_{x}\Phi_{x})\partial^k_{x}V-\mu \partial^k_{x}(V_{x}\bar{\phi}_{x}- \bar{\rho}_{x}\Phi)\partial^k_{x}V -\partial^k_{x}h_{x}\partial^k_{x}V-\partial^k_{x}f_x\partial^k_{x}V.
\end{split}
\end{equation}
By simple calculations, we have
\begin{equation*}
\begin{split}
 \partial^{k}_{x}V_{tt}\partial^k_{x}V=&\frac{d}{dt}(\partial^k_{x}V_{t}\partial^k_{x}V)-|\partial^k_{x}V_t|^2,\\
 -\partial^{k}_{x}(p'(\bar{\rho})V_x)_{x}\partial^k_{x}V=&-[\partial^{k}_{x}(p'(\bar{\rho})V_x)\partial^k_{x}V]_x+\partial^{k}_{x}(p'(\bar{\rho})V_x)\partial^k_{x}V_x\\
 =&-[\partial^{k}_{x}(p'(\bar{\rho})V_x)\partial^k_{x}V]_x+p'(\bar{\rho})\left(\partial^k_{x}V_{x}\right)^2+
\sum_{\ell<k}C_{k}^{\ell}\partial^{k-\ell}_{x}(p'(\bar{\rho}))\partial_{x}^{\ell}V_{x}\partial_{x}^{k}V_{x}, \\
\mu \partial^{k}_{x}(\bar{\rho }\Phi)_{x}\partial^k_{x}V=& \mu[ \partial^{k}_{x}(\bar{\rho }\Phi)\partial^k_{x}V ]_x-\mu \partial^{k}_{x}(\bar{\rho }\Phi)\partial^k_{x}V_x \\
=& \mu[ \partial^{k}_{x}(\bar{\rho }\Phi)\partial^k_{x}V ]_x-\mu\bar{\rho}\partial^k_{x}\Phi\partial^k_{x}V_{x}-\mu \sum_{\ell<k}C_{k}^{\ell}\partial^{k-\ell}_{x}\bar{\rho}\partial^\ell_{x}\Phi\partial^k_{x}V_{x},\\
\mu \partial^k_{x}(V_{x}\bar{\phi}_{x}- \bar{\rho}_{x}\Phi)\partial^k_{x}V=& \frac{\mu}{b}\partial^k_{x}\left[\bar{\rho}_{x}\left(aV_{x}-b\Phi\right)\right]\partial^k_{x}V,
\end{split}
\end{equation*}
where we have used the relationship between $\bar{\phi}$ and $\bar{\rho}$, i.e., $\bar{\phi}=\frac{a}{b} \bar{\rho}$.
Substituting the above equality into \eqref{va1d-2} and integrating the equation over $\R$, we get
\begin{equation}\label{va1d2}
\begin{split}
&\frac{d}{dt}\left\{\frac{\alpha}{2}\|\partial^k_{x}V\|^2+\int_{\R}\partial^k_{x}V_{t}\partial^k_{x}Vdx\right\}+\int_{\R} p'(\bar{\rho})\left(\partial^k_{x}V_{x}\right)^2dx-\mu \int_{\R}\bar{\rho}\partial^k_{x}\Phi\partial^k_{x}V_{x}dx\\
=&\|\partial^k_{x}V_{t}\|^2-\sum_{\ell<k}C_{k}^{\ell}\int_{\R}\partial^{k-\ell}_{x}(p'(\bar{\rho}))\partial_{x}^{\ell}V_{x}\partial_{x}^{k}V_{x}dx+\mu \sum_{\ell<k}C_{k}^{\ell}\int_{\R}\partial^{k-\ell}_{x}\bar{\rho}\partial^\ell_{x}\Phi\partial^k_{x}V_{x}dx\\
&-\int_{\R}\mu\partial^k_{x}(V_{x}\Phi_{x})\partial^k_{x} Vdx-\frac{\mu}{b}\int_{\R}\partial^k_{x}\left[\bar{\rho}_{x}\left(aV_{x}-b\Phi\right)\right]\partial^k_{x}Vdx\\
&-\int_{\R}\partial^k_{x}h_{x}\partial^k_{x} Vdx-\int_{\R}\partial^k_{x}f_{x}\partial^k_{x}Vdx.
\end{split}
\end{equation}
Applying  $\partial^{k}_{x}$ for $0\leq k\leq 2$ to the second equation of $\eqref{va1d1fxk}$, multiplying the result by $\frac{\mu \bar{\rho}}{a}\partial^k_{x}\Phi$, and integrating the resulting equation with respect to $x$ give
\begin{equation}\label{va1dphi}
	\begin{split}
		&\frac{\mu}{2a}\frac{d}{dt}\int_{\R}\bar{\rho}\left(\partial^k_{x}\Phi\right)^2dx+\frac{\mu D}{a}\int_{\R} \bar{\rho}\left(\partial^k_{x}\Phi_{x}\right)^2dx+\frac{b\mu}{a}  \int_{\R}\bar{\rho}\left(\partial^k_{x}\Phi\right)^2dx-\mu \int_{\R}\partial^k_{x}V_{x}\bar{\rho}\partial^k_{x}\Phi dx\\
		=&\frac{\mu}{2a}\int_{\R}\bar{\rho}_{t}\left(\partial^k_{x}\Phi\right)^2 dx-\frac{\mu D}{a}\int_{\R} \bar{\rho}_{x}\partial^k_{x}\Phi_{x}\partial^k_{x}\Phi dx+\frac{\mu}{a}\int_{\R} \bar{\rho}\partial^k_{x}\Phi\partial_{x}^{k}g dx.
	\end{split}
\end{equation}
Taking summation of \eqref{va1d2} and \eqref{va1dphi} and using integration by parts for the last two terms of \eqref{va1d2}, we have
\begin{equation}\label{va1d2sum}
	\begin{split}
		&\frac{d}{dt}\left\{\frac{\alpha}{2}\|\partial^k_{x}V\|^2+\int_{\R}\partial^k_{x}V_{t}\partial^k_{x}Vdx+\frac{\mu}{2a}\int_{\R}\bar{\rho}\left(\partial^k_{x}\Phi\right)^2dx\right\}+\frac{\mu D}{a}\int_{\R} \bar{\rho}\left(\partial^k_{x}\Phi_{x}\right)^2dx\\
		&+\int_{\R} p'(\bar{\rho})\left(\partial^k_{x}V_{x}\right)^2dx-2\mu \int_{\R}\bar{\rho}\partial^k_{x}\Phi\partial^k_{x}V_{x}dx+\frac{b\mu}{a}  \int_{\R}\bar{\rho}\left(\partial^k_{x}\Phi\right)^2dx\\
		=&\|\partial^k_{x}V_{t}\|^2-\int_{\R}\mu\partial^k_{x}(V_{x}\Phi_{x})\partial^k_{x} Vdx-\frac{\mu}{b}\int_{\R}\partial^k_{x}\left[\bar{\rho}_{x}\left(aV_{x}-b\Phi\right)\right]\partial^k_{x}Vdx\\
		&+\int_{\R}\partial^k_{x}h\partial^k_{x} V_{x}dx+\int_{\R}\partial^k_{x}f\partial^k_{x}V_{x}dx+\frac{\mu}{2a}\int_{\R}\bar{\rho}_{t}\left(\partial^k_{x}\Phi\right)^2dx-\frac{\mu D}{a}\int_{\R} \bar{\rho}_{x}\partial^k_{x}\Phi_{x}\partial^k_{x}\Phi dx\\
		&+\frac{\mu}{a}\int_{\R} \bar{\rho}\partial^k_{x}\Phi\partial_{x}^{k}g dx+\sum_{\ell<k}C_{k}^{\ell}I_{k,\ell}\\
		=:&\|\partial^k_{x}V_{t}\|^2+\sum_{j=1}^{7}I_{j}+\sum_{\ell<k}C_{k}^{\ell}I_{k,\ell},
	\end{split}
\end{equation}
with
$$
I_{k,\ell}=-\int_{\R}\partial^{k-1-\ell}_{x}(p''(\bar{\rho})\bar{\rho}_{x})\partial_{x}^{\ell}V_{x}\partial_{x}^{k}V_{x}dx+\mu \int_{\R}\partial^{k-1-\ell}_{x}\bar{\rho}_{x}\partial^\ell_{x}\Phi\partial^k_{x}V_{x}dx,
$$
where the fact $\ell<k$ has been used.

When $k=0$, we estimate $I_{1}$-$I_{7}$ as follows.
By using the Cauchy-Schwartz inequality, Sobolev inequality and \eqref{priori}, $I_{1}$ can be estimated as follows,
\begin{equation*}
	I_{1}\leq C\|V\|_{L^{\infty}}\|V_{x}\|\|\Phi_{x}\|\leq CN(T)\left(\|V_{x}\|^2+\|\Phi_{x}\|^2\right).
\end{equation*}
For $I_{2}$, we have from the second equation of \eqref{va1d1fxk} and Lemma \ref{prodecay} that
\begin{equation*}
	\begin{split}
	I_{2}=&-\frac{\mu}{b}\int_{\R}\bar{\rho}_{x}\left(aV_{x}-b\Phi\right)Vdx\\
	     =&-\frac{\mu}{b}\int_{\R}\bar{\rho}_{x}(\Phi_{t}-D\Phi_{xx}-g)Vdx\\
	     =&-\frac{\mu}{b}\frac{d}{dt}\int_{\R}\bar{\rho}_{x}\Phi Vdx+\frac{\mu}{b}\int_{\R}\bar{\rho}_{xt}\Phi Vdx+\frac{\mu}{b}\int_{\R}\bar{\rho}_{x}\Phi V_{t}dx\\
	     &-\frac{\mu D}{b}\int_{\R}\bar{\rho}_{xx}\Phi_{x}Vdx-\frac{\mu D}{b}\int_{\R}\bar{\rho}_{x}\Phi_{x}V_{x}dx+\frac{\mu}{b}\int_{\R}\bar{\rho}_{x}(-\bar{\phi}_{t}+D\bar{\phi}_{xx}) Vdx\\
	     \leq &-\frac{\mu}{b}\frac{d}{dt}\int_{\R}\bar{\rho}_{x}\Phi Vdx+C\left(\|\bar{\rho}_{xt}\|^2+\|\bar{\rho}_{xx}\|^2\right)+C\|V\|_{L^{\infty}}^{2}
	     \left(\|\Phi\|^2+\|\Phi_{x}\|^2\right)\\
	     &+C\|\bar{\rho}_{x}\|_{L^{\infty}}\left(\|\Phi\|^2+\|V_{t}\|^2+\|\Phi_{x}\|^2+\|V_{x}\|^2\right)+C\|V\|\|\bar{\rho}_{x}\|_{L^{\infty}}\left(\|\bar{\phi}_{t}\|+\|\bar{\phi}_{xx}\|\right)\\
	     \leq & -\frac{\mu}{b}\frac{d}{dt}\int_{\R}\bar{\rho}_{x}\Phi Vdx +C\delta_{0}(1+t)^{-\frac{3}{2}}+CN(T)\delta_{0}(1+t)^{-\frac{5}{4}}\\
	    &+C(N(T)+\delta_{0})\left(\|\Phi\|^2+\|V_{t}\|^2+\|\Phi_{x}\|^2+\|V_{x}\|^2\right).
	\end{split}
\end{equation*}
Recall the definition of $h$ and $f$ in \eqref{Defhf} and \eqref{bound1}, we have
 \begin{equation*}
 			h\sim V_{t}^2+\bar{\rho}_{x}^2, \ \ f\sim \bar{\rho}_{t}+V_{x}^2,
\end{equation*}
 which implies that
 \begin{equation*}
 \begin{split}
I_{3}+I_{4} \leq& C\int_{\R} |V_{t}^2+\bar{\rho}_{x}^2||V_x|dx+C\int_{\R} |\bar{\rho}_{t}+V_{x}^2||V_x|dx\\
\le& C(N(T)+\delta_{0})\left(\|V_{t}\|^{2}+\|V_{x}\|^2\right)+C\delta_{0}(1+t)^{-\frac{3}{2}}.
\end{split}
 \end{equation*}
Moreover, by using Cauchy-Schwartz inequality and Lemma \ref{prodecay}, the other three terms can be estimated as
\begin{equation*}
I_{5}+I_{6}+I_{7}\leq C(N(T)+\delta_{0})\left(\|\Phi\|^{2}+\|\Phi_{x}\|^2\right)+C\delta_{0}(1+t)^{-\frac{3}{2}}.
\end{equation*}

When $1\leq k\leq 2$, we estimate $I_{1}$-$I_{7}$ term by term. By Lemma \ref{lemma3.1}, it is easy to have
\begin{equation*}
	\begin{split}
	I_{1}\leq &C\left(\|V_{x}\|_{L^{\infty}}\|\partial_{x}^{k}\Phi_{x}\|+\|\Phi_{x}\|_{L^{\infty}}\|\partial_{x}^{k}V_{x}\|\right)\|\partial_{x}^{k}V\|\\
	\leq & CN(T)\left(\|\partial_{x}^{k}\Phi_{x}\|^2+\|\partial_{x}^{k}V_{x}\|^2+\|\partial_{x}^{k}V\|^2\right)
	\end{split}
\end{equation*}
and
\begin{equation*}
	\begin{split}
		I_{2}=&-\frac{\mu}{b}\int_{\R}\partial_{x}^{k}\left[\bar{\rho}_{x}\left(aV_{x}-b\Phi\right)\right] \partial_{x}^{k}Vdx\\
		\leq & C\left(\|\bar{\rho}_{x}\|_{L^{\infty}}\|aV_{x}-b\Phi\|+\|\left(aV_{x}-b\Phi\right)\|_{L^\infty}\|\partial_{x}^{k}\bar{\rho}_{x}\|\right)\|\partial_{x}^{k}V\|\\
		\leq &C\delta_{0}\left(\|\Phi\|_{2}^2+\|V_{x}\|_{2}^2\right).
	\end{split}
\end{equation*}
For $I_{3}$, $I_{4}$, recalling the definition of $h$ and $f$ in \eqref{Defhf} and using \eqref{priori},  Cauchy-Schwartz inequality and Lemma \ref{prodecay},  we can derive that
\begin{equation*}
	I_{3}+I_{4}\le C(N(T)+\delta_{0})\left(\|V_{t}\|_{2}^{2}+\|V_{x}\|_{2}^2\right)+C\delta_{0}(1+t)^{-\frac{3}{2}-k}.
\end{equation*}
Due to the properties of $\bar{\rho}_{t}$ and $\bar{\rho}_{x}$ in Lemma \ref{prodecay}, $I_{5}$, $I_{6}$ and $I_{7}$ can be estimated as follows.
\begin{equation*}
I_{5}+I_{6}+I_{7}\leq C\delta_{0}\left(\|\partial^k_{x}\Phi\|^{2}+\|\partial^k_{x}\Phi_{x}\|^2\right)+C\delta_{0}(1+t)^{-\frac{3}{2}-k}.
\end{equation*}
In order to complete step 1, we still need to estimate the last term $I_{k,\ell}$. It follows from \eqref{sobin}, \eqref{priori} and Lemma \ref{prodecay} that
\begin{equation*}
		\begin{split}
	|I_{k,\ell}|=&\left|-\int_{\R}\partial^{k-1-\ell}_{x}(p''(\bar{\rho})\bar{\rho}_{x})\partial_{x}^{\ell}V_{x} \partial_{x}^{k}V_{x}dx+\mu \int_{\R}\partial^{k-1-\ell}_{x}\bar{\rho}_{x}\partial^\ell_{x}\Phi \partial^k_{x}V_{x}dx\right|\\
	\leq & C(\|\bar{\rho}_{x}\|_{L^{\infty}}\|\partial_{x}^{k-1}V_{x}\|+\|V_{x}\|_{L^{\infty}}\|\partial_{x}^{k-1}\bar{\rho}_{x}\|)\|\partial^kV_{x}\|\\
	&+C(\|\bar{\rho}_{x}\|_{L^{\infty}}\|\partial_{x}^{k-1}\Phi\|+\|\Phi\|_{L^{\infty}}\|\partial_{x}^{k-1}\bar{\rho}_{x}\|)\|\partial^kV_{x}\|\\
	\leq & C\delta_{0}\left(\|V_{x}\|_{2}^2+\|\Phi\|_{2}^2\right).
		\end{split}
	\end{equation*}
Then $\eqref{Step1C}$ is obtained  by substituting all the estimates $I_{1}$-$I_{7}$ and $I_{k,\ell}$ into \eqref{va1d2sum} and then taking summation over $0\leq k\leq 2$.

\medskip

{\bf{Step 2.}}\ It is easy to check that
\begin{equation}\label{Step2C}
	\begin{aligned}
			&\frac{d}{dt}\sum_{0\leq k\leq 2}\left\{\frac{1}{2}\|\partial^k_{x}V_{t}\|^2+\frac{\mu D}{2a}\int_{\R} \bar{\rho}\left(\partial^k_{x}\Phi_{x}\right)^2dx\right\}+\alpha\|V_{t}\|_{2}^2+\frac{\mu}{a}\sum_{0\leq k\leq 2}\int_{\R}\bar{\rho}\left(\partial^k_{x}\Phi_{t} \right)^2dx \\[2mm]
		&+\frac{1}{2}\frac{d}{dt}\sum_{0\leq k\leq 2}\left\{\int_{\R} p'(\bar{\rho})\left(\partial^k_{x}V_{x}\right)^2dx-2\mu \int_{\R}\bar{\rho}\partial^k_{x}\Phi \partial^k_{x}V_{x}dx+\frac{\mu b}{a}\int_{\R}\bar{\rho}\left( \partial^k_{x}\Phi\right)^2dx\right\}\\[2mm]
	\leq & \frac{1}{2}\frac{d}{dt}\sum_{0\leq k\leq 2}\left\{\int_{\R}\frac{(V_t+\frac{1}{\alpha}(q(\bar{\rho}))_x)^2}{(V_x+\bar{\rho})^2}\left(\partial^k_{x}V_{x}\right)^2dx-
\int_{\R}[p'(V_x+\bar{\rho})-p'(\bar{\rho})]\left(\partial^k_{x}V_{x}\right)^2dx\right\}\\
	&+C(\delta_{0}+N(T))\left(\|V_{x}\|_{2}^2+\|V_{t}\|_{2}^2+\|\Phi\|_{3}^2+\|\Phi_{t}\|_{2}^2\right) +C\delta_{0}(1+t)^{-\frac{3}{2}}.
	\end{aligned}
\end{equation}
Applying $\partial^{k}_{x}$ for $0\leq k\leq 2$ to the first equation of $\eqref{va1d1fxk}$ and multiplying it by $\partial^k_{x}V_{t}$, we get
\begin{equation}\label{va1d-3}
\begin{split}
& \partial^{k}_{x} V_{tt}\partial^k_{x}V_t-\partial^{k}_{x}(p'(\bar{\rho})V_x)_{x}\partial^k_{x}V_t+\alpha \partial^{k}_{x}V_{t}\partial^k_{x}V_t+\mu \partial^{k}_{x}(\bar{\rho }\Phi)_{x}\partial^k_{x}V_t\\
=& -\mu \partial^k_{x}(V_{x}\Phi_{x})\partial^k_{x}V_t-\mu \partial^k_{x}(V_{x}\bar{\phi}_{x}- \bar{\rho}_{x}\Phi)\partial^k_{x}V_t -\partial^k_{x}h_{x}\partial^k_{x}V_t-\partial^k_{x}f_x\partial^k_{x}V_t.
\end{split}
\end{equation}
Noting that
\begin{equation*}
\begin{split}
\partial^{k}_{x}V_{tt}\partial^k_{x}V_t=&\frac12\frac{d}{dt}|\partial^k_{x}V_{t}|^2,\\[2mm]
-\partial^{k}_{x}(p'(\bar{\rho})V_x)_{x}\partial^k_{x}V_t=&-\partial^{k}_{x}(p''(\bar{\rho})\bar{\rho}_xV_x)\partial^k_{x}V_t-\partial^{k}_{x}(p'(\bar{\rho})V_{xx})\partial^k_{x}V_t\\
=&-\partial^{k}_{x}(p''(\bar{\rho})\bar{\rho}_xV_x)\partial^k_{x}V_t-p'(\bar{\rho})\partial^{k}_{x}V_{xx}\partial^k_{x}V_t-\sum_{\ell<k}C_{k}^{\ell}\partial^{k-\ell}_{x}(p'(\bar{\rho}))\partial_{x}^{\ell}V_{xx} \partial_{x}^{k}V_{t}\\
=&-\left(p'(\bar{\rho})\partial^{k}_{x}V_{x}\partial^k_{x}V_t\right)_x+p''(\bar{\rho})\bar{\rho}_x\partial^{k}_{x}V_{x}\partial^k_{x}V_t+p'(\bar{\rho})\partial^{k}_{x}V_{x}\partial^k_{x}V_{xt}\\
&-\partial^{k}_{x}(p''(\bar{\rho})\bar{\rho}_xV_x)\partial^k_{x}V_t-\sum_{\ell<k}C_{k}^{\ell}\partial^{k-\ell}_{x}(p'(\bar{\rho}))\partial_{x}^{\ell}V_{xx} \partial_{x}^{k}V_{t}\\
=&-\left(p'(\bar{\rho})\partial^{k}_{x}V_{x}\partial^k_{x}V_t\right)_x+p''(\bar{\rho})\bar{\rho}_x\partial^{k}_{x}V_{x}\partial^k_{x}V_t+\frac{1}{2}\frac{d}{dt}\left(p'(\bar{\rho})(\partial^k_{x}V_{x})^2\right)\\
&-\frac{1}{2} p''(\bar{\rho})\bar{\rho}_{t}\left(\partial^k_{x}V_{x}\right)^2-\partial^{k}_{x}(p''(\bar{\rho})\bar{\rho}_xV_x)\partial^k_{x}V_t-\sum_{\ell<k}C_{k}^{\ell}\partial^{k-\ell}_{x}(p'(\bar{\rho}))\partial_{x}^{\ell}V_{xx} \partial_{x}^{k}V_{t},
\end{split}
\end{equation*}
\begin{equation*}
\begin{split}
\mu \partial^{k}_{x}(\bar{\rho }\Phi)_{x}\partial^k_{x}V_t=& \mu[ \partial^{k}_{x}(\bar{\rho }\Phi)\partial^k_{x}V_t ]_x-\mu \partial^{k}_{x}(\bar{\rho }\Phi)\partial^k_{x}V_{xt} \\
=& \mu[ \partial^{k}_{x}(\bar{\rho }\Phi)\partial^k_{x}V_t ]_{x}-\mu\bar{\rho}\partial^k_{x}\Phi\partial^k_{x}V_{xt}
-\mu\sum_{\ell<k}C_{k}^{\ell}\partial^{k-\ell}_{x}\bar{\rho}\partial^\ell_{x}\Phi\partial^k_{x}V_{xt}\\
=& \mu[ \partial^{k}_{x}(\bar{\rho }\Phi)\partial^k_{x}V_t ]_{x}-\mu\frac{d}{dt}\left(\bar{\rho}\partial^k_{x}\Phi\partial^k_{x}V_{x}\right)+\mu\bar{\rho}_t\partial^k_{x}\Phi\partial^k_{x}V_{x}+\mu\bar{\rho}\partial^k_{x}\Phi_t\partial^k_{x}V_{x}\\
&-\mu\sum_{\ell<k}C_{k}^{\ell}\left(\partial^{k-\ell}_{x}\bar{\rho}\partial^\ell_{x}\Phi\partial^k_{x}V_{t}\right)_{x}+\mu\sum_{\ell<k}C_{k}^{\ell}\left(\partial^{k-\ell}_{x}\bar{\rho}_{x}\partial^\ell_{x}\Phi\partial^k_{x}V_{t}+\partial^{k-\ell}_{x}\bar{\rho}\partial^\ell_{x}\Phi_{x}\partial^k_{x}V_{t}\right),\\
\mu \partial^k_{x}(V_{x}\bar{\phi}_{x}- \bar{\rho}_{x}\Phi)\partial^k_{x}V_t
=& \frac{\mu}{b}\partial^k_{x}\left[\bar{\rho}_{x}\left(aV_{x}-b\Phi\right)\right]\partial^k_{x}V_t.
\end{split}
\end{equation*}
Substituting the above equalities into \eqref{va1d-3} and integrating the equation over $\R$, we have
\begin{equation}\label{va1d2vt}
	\begin{split}
		&\frac{d}{dt}\left\{\frac{1}{2}\|\partial^k_{x}V_{t}\|^2+\frac{1}{2}\int_{\R} p'(\bar{\rho})\left(\partial^k_{x}V_{x} \right)^2dx-\mu \int_{\R}\bar{\rho}\partial^k_{x}\Phi \partial^k_{x}V_{x}dx\right\}+\alpha\|\partial^k_{x}V_{t}\|^2+\mu\int_{\R} \bar{\rho}\partial^k_{x}\Phi_{t} \partial^k_{x}V_{x}dx\\[2mm]
		=&\frac{1}{2}\int_{\R} p''(\bar{\rho})\bar{\rho}_{t}\left(\partial^k_{x}V_{x}\right)^2dx-\int_{\R} p''(\bar{\rho})\bar{\rho}_{x}\partial^k_{x}V_{x}\partial^k_{x}V_{t}dx+\int_{\R} \partial_{x}^{k}\left(p''(\bar{\rho})\bar{\rho}_{x}V_{x}\right)\partial^k_{x}V_{t}dx\\[2mm]
		&+\sum_{\ell<k}C_{k}^{\ell}\int_{\R}\partial^{k-\ell}_{x}(p'(\bar{\rho}))\partial_{x}^{\ell}V_{xx} \partial_{x}^{k}V_{t}dx-\mu\sum_{\ell<k}C_{k}^{\ell} \int_{\R}\left(\partial^{k-\ell}_{x}\bar{\rho}_{x}\partial^\ell_{x}\Phi\partial^k_{x}V_{t}+\partial^{k-\ell}_{x}\bar{\rho}\partial^\ell_{x}\Phi_{x}\partial^k_{x}V_{t}\right)dx\\[2mm]
		&-\mu \int_{\R}\bar{\rho}_{t}\partial^k_{x}\Phi \partial^k_{x}V_{x}dx-\int_{\R}\mu\partial^k_{x}(V_{x}\Phi_{x})\partial^k_{x} V_{t}dx-\frac{\mu}{b}\int_{\R}\partial^k_{x}\left[\bar{\rho}_{x}\left(aV_{x}-b\Phi\right)\right] \partial^k_{x}V_{t}dx\\
		&-\int_{\R}\partial^k_{x}h_{x} \partial^k_{x} V_{t}dx-\int_{\R}\partial^k_{x}f_{x} \partial^k_{x}V_{t}dx.
	\end{split}
\end{equation}
Applying $\partial^{k}_{x}$ for $0\leq k\leq 2$ to the second equation of $\eqref{va1d1fxk}$, multiplying the resultant equation  by $\frac{\mu}{a}\bar{\rho}\partial^k_{x}\Phi_{t}$ and taking integration in $x$ give

\begin{equation}\label{va1dphit}
	\begin{split}
		&\frac{d}{dt}\left\{\frac{\mu b}{2a}\int_{\R}\bar{\rho}\left(\partial^k_{x}\Phi\right)^2dx+\frac{\mu D}{2a}\int_{\R} \bar{\rho}\left(\partial^k_{x}\Phi_{x}\right)^2dx\right\}+\frac{\mu}{a}\int_{\R}\bar{\rho}\left(\partial^k_{x}\Phi_{t} \right)^2dx-\mu \int_{\R}\partial^k_{x}V_{x} \bar{\rho}\partial^k_{x}\Phi_{t}dx\\[2mm]
		=&\frac{\mu b}{2a}\int_{\R}\bar{\rho}_{t}\left(\partial^k_{x}\Phi\right)^2dx+\frac{\mu D}{2a}\int_{\R} \bar{\rho}_{t}\left(\partial^k_{x}\Phi_{x}\right)^2dx-\frac{\mu D}{a}\int_{\R} \bar{\rho}_{x}\partial^k_{x}\Phi_{x} \partial^k_{x}\Phi_{t}dx+\frac{\mu}{a}\int_{\R} \bar{\rho}\partial^k_{x}\Phi_{t} \partial_{x}^{k}gdx.
	\end{split}
\end{equation}
Combining \eqref{va1d2vt} with \eqref{va1dphit} yields
\begin{equation}\label{va1d2sumt}
\begin{split}
	&\frac{d}{dt}\left\{\frac{1}{2}\|\partial^k_{x}V_{t}\|^2+\frac{\mu D}{2a}\int_{\R} \bar{\rho}\left(\partial^k_{x}\Phi_{x}\right)^2dx\right\}+\alpha\|\partial^k_{x}V_{t}\|^2+\frac{\mu}{a}\int_{\R}\bar{\rho}\left(\partial^k_{x}\Phi_{t} \right)^2dx \\[2mm]
	&+\frac{1}{2}\frac{d}{dt}\left\{\int_{\R} p'(\bar{\rho})\left(\partial^k_{x}V_{x}\right)^2dx-2\mu \int_{\R}\bar{\rho}\partial^k_{x}\Phi \partial^k_{x}V_{x}dx+\frac{\mu b}{a}\int_{\R}\bar{\rho}\left( \partial^k_{x}\Phi\right)^2dx\right\}\\[2mm]
	&	=:J_{1}^{(k)}+J_{2}^{(k)}+J_{3}^{(k)}+J_{4}^{(k)}+\sum_{\ell<k}C_{k}^{\ell}J_{k,\ell},
\end{split}
\end{equation}
where
\begin{equation*}
	\begin{split}
		J_{1}^{(k)}=&\frac{1}{2}\int_{\R} p''(\bar{\rho})\bar{\rho}_{t}\left(\partial^k_{x}V_{x}\right)^2dx-\int_{\R} p''(\bar{\rho})\bar{\rho}_{x}\partial^k_{x}V_{x} \partial^k_{x}V_{t}dx -\mu \int_{\R}\bar{\rho}_{t}\partial^k_{x}\Phi \partial^k_{x}V_{x}dx\\
		&+\frac{\mu b}{2a}\int_{\R}\bar{\rho}_{t}\left( \partial^k_{x}\Phi\right)^2dx+\frac{\mu D}{2a}\int_{\R} \bar{\rho}_{t}\left(\partial^k_{x}\Phi_{x}\right)^2dx-\frac{\mu D}{a}\int_{\R} \bar{\rho}_{x}\partial^k_{x}\Phi_{x} \partial^k_{x}\Phi_{t}dx,
	\end{split}
\end{equation*}
\begin{equation*}
	\begin{split}
		J_{2}^{(k)}=\int_{\R} \partial_{x}^{k}\left(p''(\bar{\rho})\bar{\rho}_{x}V_{x}\right) \partial^k_{x}V_{t}dx
		-\int_{\R}\mu\partial^k_{x}(V_{x}\Phi_{x}) \partial^k_{x} V_{t}dx-\frac{\mu}{b}\int_{\R}\partial^k_{x}\left[\bar{\rho}_{x}\left(aV_{x}-b\Phi\right)\right] \partial^k_{x}V_{t}dx,
	\end{split}
\end{equation*}
\begin{equation*}
\begin{split}
		J_{3}^{(k)}&=-\int_{\R}\partial^k_{x}h_{x} \partial^k_{x} V_{t}dx-\int_{\R}\partial^k_{x}f_{x} \partial^k_{x}V_{t}dx=:J_{3h}^{(k)}+J_{3f}^{(k)},\\ J_{4}^{(k)}&=\frac{\mu}{a}\int_{\R} \bar{\rho}\partial^k_{x}\Phi_{t} \partial_{x}^{k}gdx
		\end{split}
\end{equation*}
and
\begin{equation*}
	J_{k,\ell}=\int_{\R}\partial^{k-\ell}_{x}(p'(\bar{\rho}))\partial_{x}^{\ell}V_{xx} \partial_{x}^{k}V_{t}dx-\mu \int_{\R}\partial^{k-\ell}_{x}\bar{\rho}_{x}\partial^\ell_{x}\Phi \partial^k_{x}V_{t}dx-\mu\int_{\R}\partial^{k-\ell}_{x}\bar{\rho}\partial^\ell_{x}\Phi_{x} \partial^k_{x}V_{t}dx.
\end{equation*}
It follows from the Cauchy-Schwartz inequality and Lemma \ref{prodecay} that
\begin{equation*}
		J_{1}^{(k)}\leq C\delta_{0} (\|\partial^k_{x}V_{x}\|^2+\|\partial^k_{x}V_{t}\|^2+\|\partial^k_{x}\Phi\|^2+\|\partial^k_{x}\Phi_{x}\|^2+\|\partial^k_{x}\Phi_{t}\|^2).
\end{equation*}
Due to Lemma \ref{lemma3.1}, we can estimate $J^{(k)}_{2}$ and $J_{k,\ell}$ as follows,
\begin{equation*}
	\begin{split}
	J_{2}^{(k)}\leq &C\left(\|\bar{\rho}_{x}\|_{L^{\infty}}\|\partial^k_{x}[V_{x},\Phi]\|+\|[V_{x},\Phi]\|_{L^{\infty}}\|\partial^k_{x}\bar{\rho}_{x} \|\right)\|\partial^k_{x}V_{t}\|\\
	&+\left(\|V_{x}\|_{L^{\infty}}\|\partial^k_{x}\Phi_{x}\|+\|\Phi_{x}\|_{L^{\infty}}\|\partial^k_{x}V_{x} \|\right)\|\partial^k_{x}V_{t}\|\\
	\leq & C(N(T)+\delta_{0})(\|V_{x}\|_{2}^2+\|\partial^k_{x}V_{t}\|^2+\|\Phi\|_{3}^2).
	\end{split}
\end{equation*}
Thanks to $\ell<k$, $J_{k,\ell}$ can be rewritten in the following form
\begin{equation*}
	\begin{split}
		J_{k,\ell}=&-\int_{\R}\partial^{k-1-\ell}_{x}(p''(\bar{\rho})\bar{\rho}_{x})\partial_{x}^{\ell}V_{xx} \partial_{x}^{k}V_{t}dx-\mu \int_{\R}\partial^{k-\ell}_{x}\bar{\rho}_{x}\partial^\ell_{x}\Phi \partial^k_{x}V_{t}dx-\mu\int_{\R}\partial^{k-1-\ell}_{x}\bar{\rho}_{x}\partial^\ell_{x}\Phi_{x} \partial^k_{x}V_{t}dx\\
		\leq &C\left(\|\bar{\rho}_{x}\|_{L^{\infty}}\|\partial^{k-1}_{x}[V_{xx}\Phi_{x}]\|+\|[V_{xx},\Phi_{x}]\|_{L^{\infty}}\|\partial^{k-1}_{x}\bar{\rho}_{x} \|\right)\|\partial^k_{x}V_{t}\|\\
		&+\left(\|\bar{\rho}_{x}\|_{L^{\infty}}\|\partial^k_{x}\Phi\|+\|\Phi\|_{L^{\infty}}\|\partial^k_{x}\bar{\rho}_{x} \|\right)\|\partial^k_{x}V_{t}\|\\
		\leq & C\delta_{0}(\|V_{x}\|_{2}^2+\|\partial^k_{x}V_{t}\|^2+\|\Phi\|_{3}^2).
	\end{split}
\end{equation*}
For $J_{4}^{(k)}$, with the definition of $g$ in \eqref{Defhf}, using the Cauchy-Schwartz inequality and Lemma \ref{prodecay}, we have
\begin{equation*}
	\begin{split}
		J_{4}^{(k)}\leq C\delta_{0}\|\partial_{x}^{k}\Phi_{t}\|^2+\frac{C}{\delta_{0}}\|\partial_{x}^{k}g\|^{2}
		\leq  C\delta_{0}\|\partial_{x}^{k}\Phi_{t}\|^2+C\delta_{0}(1+t)^{-\frac{3}{2}-k}.
	\end{split}
\end{equation*}
In order to estimate $J_{3}^{(k)}$, we calculate $h_{x}$ and $f_{x}$ first.
Note that
\begin{equation}\label{va1dh}
	\arraycolsep=1.5pt
	\begin{array}{rl}
		h_{x}= \displaystyle-\left(\frac{(V_t+\frac{1}{\alpha}q(\bar{\rho})_x)^2}{V_x+\bar{\rho}}\right)_x
		=& \displaystyle \frac{(V_t+\frac{1}{\alpha}(q(\bar{\rho}))_x)^2}{(V_x+\bar{\rho})^2}V_{xx} -\frac{2(V_t+\frac{1}{\alpha}q(\bar{\rho})_x)}{V_x+\bar{\rho}}V_{xt}\\[3mm]
		&  \displaystyle+\frac{(V_t+\frac{1}{\alpha}(q(\bar{\rho}))_x)^2}{(V_x+\bar{\rho})^2}\bar{\rho}_{x} -\frac{2(V_t+\frac{1}{\alpha}q(\bar{\rho})_x)}{V_x+\bar{\rho}}\frac{1}{\alpha}q(\bar{\rho})_{xx},
	\end{array}
\end{equation}
and
\begin{equation}\label{va1d5d}
	\begin{split}
		f_{x}=&\frac{1}{\alpha}(q(\bar{\rho}))_{xt}-\left[p(V_x+\bar{\rho})-p(\bar{\rho})-p'(\bar{\rho})V_x\right]_x\\[2mm]
		=&\frac{1}{\alpha}(q(\bar{\rho}))_{xt}-[p'(V_x+\bar{\rho})-p'(\bar{\rho})]V_{xx}-(p'(V_x+\bar{\rho})-p'(\bar{\rho})-p''(\bar{\rho})V_x)\bar{\rho}_{x}.
	\end{split}
\end{equation}
Then, for $J_{3h}^{(k)}$, we  focus on the term which contains the highest-order derivative $V_{xx}$ and $V_{xt}$,
\begin{equation}\label{JH}
	\begin{split}
		J_{3h}^{(k)}=-&\int_{\R}\partial^k_{x}h_{x} \partial^k_{x} V_{t}dx\\
		=&-\int_{\R}\frac{(V_t+\frac{1}{\alpha}(q(\bar{\rho}))_x)^2}{(V_x+\bar{\rho})^2}\partial^k_{x}V_{xx}\partial^k_{x} V_{t}dx+\int_{\R} \frac{2(V_t+\frac{1}{\alpha}q(\bar{\rho})_x)}{V_x+\bar{\rho}}\partial^k_{x}V_{xt}\partial^k_{x} V_{t}dx\\
&\ \ -\int_{\R}({\rm O.T.H.})\partial^k_{x} V_{t} dx\\
		=&\int_{\R}\left(\frac{(V_t+\frac{1}{\alpha}(q(\bar{\rho}))_x)^2}{(V_x+\bar{\rho})^2}\right)_{x}\partial^k_{x}V_{x}\partial^k_{x} V_{t}dx+\frac{1}{2}\frac{d}{dt}\int_{\R}\frac{(V_t+\frac{1}{\alpha}(q(\bar{\rho}))_x)^2}{(V_x+\bar{\rho})^2}\left(\partial^k_{x}V_{x}\right)^2dx\\
	   &-\frac{1}{2}\int_{\R}\left(\frac{(V_t+\frac{1}{\alpha}(q(\bar{\rho}))_x)^2}{(V_x+\bar{\rho})^2}\right)_{t}\left(\partial^k_{x}V_{x}\right)^2dx-\int_{\R}\left( \frac{(V_t+\frac{1}{\alpha}q(\bar{\rho})_x)}{V_x+\bar{\rho}}\right)_{x}\left(\partial^k_{x}V_{t}\right)^2dx\\
	   	&-\int_{\R}({\rm O.T.H.})\partial^k_{x} V_{t} dx,
			\end{split}
\end{equation}
where $({\rm O.T.H.})$ as an abbreviation for ``other terms of $\partial_{x}^{k}h_{x}$" reading as
\begin{equation*}
	\begin{split}
	({\rm O.T.H.})=&\sum_{\ell<k}C_{k}^{\ell}\partial_{x}^{k-\ell}\left(\frac{(V_t+\frac{1}{\alpha}(q(\bar{\rho}))_x)^2}{(V_x+\bar{\rho})^2}\right)\partial^{\ell}_{x}V_{xx}-\sum_{\ell<k}C_{k}^{\ell} \partial_{x}^{k-\ell}\left(\frac{2(V_t+\frac{1}{\alpha}q(\bar{\rho})_x)}{V_x+\bar{\rho}}\right)\partial^\ell_{x}V_{xt}\\
&+\partial_{x}^{k}\left(\frac{(V_t+\frac{1}{\alpha}(q(\bar{\rho}))_x)^2}{(V_x+\bar{\rho})^2}\bar{\rho}_{x} -\frac{2(V_t+\frac{1}{\alpha}q(\bar{\rho})_x)}{V_x+\bar{\rho}}\frac{1}{\alpha}q(\bar{\rho})_{xx}\right).
	\end{split}
\end{equation*}
Similarly, for $J_{3f}^{(k)}$, we focus on the term which contains the highest-order derivative $V_{xx}$,
\begin{equation}\label{JF}
	\begin{split}
		J_{3f}^{(k)}=&-\int_{\R}\partial^k_{x}f_{x} \partial^k_{x} V_{t}dx
		=\int_{\R}[p'(V_x+\bar{\rho})-p'(\bar{\rho})]\partial^k_{x}V_{xx} \partial^k_{x} V_{t}dx-\int_{\R}({\rm O.T.F.})\partial^k_{x} V_{t} dx,\\
		=&-\int_{\R}[p'(V_x+\bar{\rho})-p'(\bar{\rho})]\partial^k_{x}V_{x} \partial^k_{x} V_{xt}dx-\int_{\R}[p'(V_x-\bar{\rho})-p'(\bar{\rho})]_{x}\partial^k_{x}V_{x} \partial^k_{x} V_{t}dx\\
&\ \ -\int_{\R}({\rm O.T.F.})\partial^k_{x} V_{t} dx\\		=&-\frac{1}{2}\frac{d}{dt}\int_{\R}[p'(V_x+\bar{\rho})-p'(\bar{\rho})]\left(\partial^k_{x}V_{x}\right)^2dx+\frac{1}{2}\int_{\R}[p'(V_x+\bar{\rho})-p'(\bar{\rho})]_{t}\left(\partial^k_{x}V_{x}\right)^2dx\\
		&-\int_{\R}[p'(V_x+\bar{\rho})-p'(\bar{\rho})]_{x}\partial^k_{x}V_{x} \partial^k_{x} V_{t}dx-\int_{\R}({\rm O.T.F.})\partial^k_{x} V_{t} dx,
	\end{split}
\end{equation}	
with
\begin{equation*}
	\begin{split}
		({\rm O.T.F.})=	&
		-\sum_{\ell<k}C_{k}^{\ell}\partial_{x}^{k-\ell}[p'(V_x+\bar{\rho})-p'(\bar{\rho})]\partial^\ell_{x}V_{xx} \\
		&+\partial_{x}^{k}\left(\frac{1}{\alpha}(q(\bar{\rho}))_{xt}-(p'(V_x+\bar{\rho})-p'(\bar{\rho})-p''(\bar{\rho})V_x)\bar{\rho}_{x}\right).
	\end{split}
\end{equation*}	
Here we use $({\rm O.T.F.})$ as an abbreviation for ``other terms of $\partial^k_{x}f_{x}$". Clearly the total order of spatial derivatives for terms which are the product between $V_{x}$ and $V_{t}$ is not greater than $3$  in $({\rm O.T.H.})$ and $({\rm O.T.F.})$, and the highest order of derivatives for $V_{x}$ and $V_{t}$ is not greater than $2$. Then we can use Lemma \ref{lemma3.1} and Cauchy-Schwartz inequality to get
\begin{equation*}
	\begin{split}
		\int_{\R}({\rm O.T.H.})\partial^k_{x} V_{t} dx+\int_{\R}({\rm O.T.F.})\partial^k_{x} V_{t} dx
		\leq  C(\delta_{0}+N(T))\left(\|V_{x}\|_{2}^2+\|V_{t}\|_{2}^2\right) +C\delta_{0}(1+t)^{-\frac{5}{2}-k},
	\end{split}
\end{equation*}
where we have used the following fact
$$
|V_{tt}|\leq C\left(|V_{xx}|+|V_{t}|+|V_{x}||\Phi_{x}|+(|\bar{\rho}_{x}|+|\bar{\phi}_{x}|)|V_{x}|+|V_{xt}|+|\Phi_{x}|+|\bar{\rho}_{x}|^3+|\bar{\rho}_{t}||\bar{\rho}_{x}|\right)\leq C(N(T)+\delta_{0}).
$$
Based on the above calculations for  $J_{3h}^{(k)}$ and $J_{3f}^{(k)}$, we can estimate $J_{3}^{(k)}$ as follows,
\begin{equation*}
	\begin{split}
		J_{3}^{(k)}\leq & \frac{1}{2}\frac{d}{dt}\int_{\R}\frac{(V_t+\frac{1}{\alpha}(q(\bar{\rho}))_x)^2}{(V_x+\bar{\rho})^2}\left(\partial^k_{x}V_{x}\right)^2dx -
		\frac{1}{2}\frac{d}{dt}\int_{\R}[p'(V_x+\bar{\rho})-p'(\bar{\rho})]\left(\partial^k_{x}V_{x}\right)^2dx\\
		&+C(\delta_{0}+N(T))\left(\|V_{x}\|_{2}^2+\|V_{t}\|_{2}^2\right) +C\delta_{0}(1+t)^{-\frac{5}{2}-k}.
	\end{split}
\end{equation*}
Feeding \eqref{va1d2sumt} on all the estimations of $J_{1}^{(k)}$-$J_{4}^{(k)}$ and $J_{k,\ell}$, and taking summation of \eqref{va1d2sumt} over $0\leq k\leq 2$, one has \eqref{Step2C}.

 \medskip

 {\bf{Step 3.}} First, under the assumptions $p'(\bar{\rho})-\frac{a\mu}{b}\bar{\rho}>0$ and $\rho_{-}<\bar{\rho}<\rho_{+}$, we have
\begin{equation}\label{equiva}
	\int_{\R} p'(\bar{\rho})\left(\partial^k_{x}V_{x}\right)^2dx-2\mu \int_{\R}\bar{\rho}\partial^k_{x}\Phi \partial^k_{x}V_{x}dx+\frac{\mu b}{a}\int_{\R}\bar{\rho}\left( \partial^k_{x}\Phi\right)^2dx\sim  \|[\partial^k_{x}\Phi,\partial^k_{x}V_{x}]\|^2.
\end{equation}
 Second, adding \eqref{Step1C} with \eqref{Step2C} multiplied by a constant $K$ which will be determined later, it follows that
\begin{equation}\label{Step3C}
	\begin{aligned}
		&\frac{d}{dt}\mathcal{E}(t)+C\left(\|V_{x}\|^2_{2}+\|\Phi\|_{3}^2+\|\Phi_{t}\|_{2}^2\right)+(\alpha K-1)\|V_{t}\|_{2}^{2}\\
		 \leq & C\delta_{0}(1+t)^{-\frac{5}{4}}+C(N(T)+\delta_{0})(\|V_{x}\|_{2}^2+\|V_{t}\|_{2}^2+\|\Phi\|_{3}^2+\|\Phi_t\|_{2}^2).
	\end{aligned}
\end{equation}
 where $\alpha K>1$, $\mathcal{E}(t)$ is given by

 \begin{equation}\label{DefE}
 	\begin{aligned}
 		\mathcal{E}(t)=&\sum_{0\leq k\leq 2}\left\{\frac{\alpha}{2}\|\partial^k_{x}V\|^2+\int_{\R}\partial^k_{x}V_{t}\partial^k_{x}Vdx+\frac{K}{2}\|\partial^k_{x}V_{t}\|^2\right\}\\
 		&+\frac{K}{2}\sum_{0\leq k\leq 2}\left\{\int_{\R} p'(\bar{\rho})\left(\partial^k_{x}V_{x}\right)^2dx-2\mu \int_{\R}\bar{\rho}\partial^k_{x}\Phi \partial^k_{x}V_{x}dx+\frac{\mu b}{a}\int_{\R}\bar{\rho}\left( \partial^k_{x}\Phi\right)^2dx\right\}\\
 		&+\sum_{0\leq k\leq 2}\left\{\frac{\mu}{2a}\int_{\R}\bar{\rho}\left(\partial^k_{x}\Phi\right)^2dx +\frac{\mu D K}{2a}\int_{\R}\bar{\rho}\left(\partial^k_{x}\Phi_{x}\right)^2dx\right\}\\[2mm]
 	&+\frac{\mu}{b}\int_{\R}\bar{\rho}_{x}\Phi Vdx -\frac{K}{2}\sum_{0\leq k\leq 2}\int_{\R}\frac{(V_t+\frac{1}{\alpha}(q(\bar{\rho}))_x)^2}{(V_x+\bar{\rho})^2}\left(\partial^k_{x}V_{x}\right)^2dx\\
 	&+\frac{K}{2}\sum_{0\leq k\leq 2}\int_{\R}[p'(V_x+\bar{\rho})-p'(\bar{\rho})]\left(\partial^k_{x}V_{x}\right)^2dx.
 	\end{aligned}
 \end{equation}
On the one hand, we can choose $K$ large enough such that
$$\frac{\alpha}{2}\|\partial^k_{x}V\|^2+\int_{\R}\partial^k_{x}V_{t}\partial^k_{x}Vdx+\frac{K}{2}\|\partial^k_{x}V_{t}\|^2\geq C\left(\|\partial^k_{x}V_{t}\|^2+\|\partial^k_{x}V\|^2\right)$$
for some  constant $C>0$ independent of $t$. On the other hand,  with the Cauchy-Schwartz inequality, Lemma \ref{prodecay} and {\it a priori} assumption, we have
\begin{equation*}
	\int_{\R}\bar{\rho}_{x}\Phi Vdx\leq C\delta_{0}\left(\|\Phi\|^2+\|V\|^2\right),
	\end{equation*}
\begin{equation*}
	\int_{\R}\frac{(V_t+\frac{1}{\alpha}(q(\bar{\rho}))_x)^2}{(V_x+\bar{\rho})^2}\left(\partial^k_{x}V_{x}\right)^2dx\leq C(N(T)+\delta_{0})\|\partial^k_{x}V_{x}\|^2
\end{equation*}
and
\begin{equation*}
	\int_{\R}[p'(V_x+\bar{\rho})-p'(\bar{\rho})]\left(\partial^k_{x}V_{x}\right)^2dx\leq CN(T)\|\partial^k_{x}V_{x}\|^2.
\end{equation*}
Choosing $N(T)$ and $\delta_{0}$ small enough, one gets from \eqref{equiva} and \eqref{DefE} that
$$
\mathcal{E}(t)\sim \|V\|_{3}^2+\|V_{t}\|_{2}^2+\|\Phi\|_{3}^2.
$$
With the above equivalence for $\mathcal{E}(t)$, after integrating \eqref{Step3C} over $(0,t)$, and taking $N(T)$ and $\delta_{0}$ sufficiently small, we get \eqref{f1}. Thus, the proof of Proposition \ref{mainpro} is completed.
\end{proof}

Finally, we need to verify that the {\it a priori} assumption \eqref{priori} is achievable. Since under the {\it a priori} assumption \eqref{priori}, we have proved that \eqref{f1} holds true when $N(T)$ is appropriately small. So as long as $\varepsilon_0$ is small enough,  \eqref{priori} is ensured by \eqref{f1}. As such under the conditions of Proposition \ref{1.1}, we close the {\it a priori} assumption \eqref{priori}.

\section{The time-decay rate of solutions}\label{sec4}

In this section, we are devoted to establishing the decay rate of the solution $(V,M,\Phi)$ or $(V,V_{t},\Phi)$ to \eqref{va1d1f}. First, we make the following  {\it a priori} assumption on $(V,V_{t},\Phi)$
\begin{equation}\label{asstime1}
\sum_{k=0}^{2}(1+t)^{k+1}\norm{\partial_{x}^k [V_{x},\Phi]}^2+\sum_{k=0}^{2}(1+t)^{k+2}\norm{\partial_{x}^k V_{t}}^2+\sum_{k=0}^{1}(1+t)^{k+3}\norm{\partial_{x}^{k}\Phi_{t}}^2\ll 1.
\end{equation}
Then we turn to prove the time-decay rate of $(V,V_{t},\Phi)$ which is indicated by the following energy estimates.
 \begin{proposition}\label{decay rate}
Under the assumption of  Proposition \ref{1.1}, we have
% \begin{equation}\label{thm1.1c1}
% \begin{split}
% &\sum_{k=0}^{3}(1+t)^{k}\norm{\partial_{x}^k V(t)}^2+\sum_{k=0}^{2}(1+t)^{k+2}\norm{\partial_{x}^k M(t)}^2+\sum_{k=0}^{2}(1+t)^{k+1}\norm{\partial_{x}^k \Phi(t)}^2\\[2mm]
% &+\sum_{k=0}^{1}(1+t)^{k+3}\norm{\partial_{x}^k \Phi_{t}(t)}^2+ \sum_{j=1}^{3}\int_{0}^t(1+s)^{j-1}\norm{\partial_{x}^{j-1} [V_{x}(s),\Phi (s)]}^2 ds\\[2mm]
% &+\int_{0}^t (1+s)\norm{M(s)}ds+\sum_{j=1}^{2}\int_{0}^t(1+s)^{j+1}\norm{\partial_{x}^{j-1} [M_{x}(s),\Phi_{t}(s)]}^2ds\\
% \leq &C\left(\norm{V_0}^2_3+\norm{M_0}^2_2+\norm{\Phi_0}^2_3+\delta_{0}\right).
% \end{split}
% \end{equation}
%\begin{equation}\label{thm1.1c1}
%	\begin{split}
%		&\sum_{k=0}^{3}(1+t)^{k}\norm{\partial_{x}^k V(t)}^2+\sum_{k=1}^{3}\int_{0}^t(1+s)^{k-1}\norm{\partial_{x}^{k}V(s)}^2 ds\\[2mm]
%		\leq &C\left(\norm{V_0}^2_3+\norm{M_0}^2_2+\norm{\Phi_0}^2_3+\delta_{0}\right).
%	\end{split}
%\end{equation}
\begin{equation}\label{thm1.1c12}
	\begin{split}
		&\sum_{k=0}^{2}(1+t)^{k+1}\norm{\partial_{x}^k [V_{x},\Phi,\Phi_{x}]}^2+\sum_{k=0}^{2}\int_{0}^t(1+s)^{k}\norm{\partial_{x}^{k}[V_{x},\Phi,\Phi_{x}](\cdot,s)}^2 ds\\[2mm]
		\leq &C\left(\norm{V_0}^2_3+\norm{M_0}^2_2+\norm{\Phi_0}^2_4+\delta_{0}\right) \leq C\epsilon^2
	\end{split}
\end{equation}
and
\begin{equation*}
	\begin{split}
		&\sum_{k=0}^{2}(1+t)^{k+2}\norm{\partial_{x}^k V_{t}}^2+\sum_{k=0}^{2}\int_{0}^t(1+s)^{k+1}\norm{\partial_{x}^k V_{t}(\cdot,s)}^2 ds\\[2mm]
			&+\sum_{k=0}^{1}(1+t)^{k+3}\norm{\partial_{x}^k [\Phi_{t},\Phi_{xt}]}^2+\sum_{k=0}^{1}\int_{0}^t(1+s)^{k+2}\norm{\partial_{x}^k [\Phi_{t},\Phi_{xt}](\cdot,s)}^2 ds\\[2mm]
		\leq &C\left(\norm{V_0}^2_3+\norm{M_0}^2_2+\norm{\Phi_0}^2_4+\delta_{0}\right) \leq C\epsilon^2.
	\end{split}
\end{equation*}

%Moreover, we have
%\begin{equation}\label{con.4.6c}
%	\begin{split}
%		&(1+t)^{5}\norm{[V_{ttt}(t),V_{ttx}(t),\Phi_{tt}(t),\Phi_{xtt}(t)]}^2 +(1+t)^4 \norm{V_{tt}(t)}^2\\[2mm]
%		&+\int_{0}^t\left((1+s)^{5}\norm{ [V_{ttt}(s),\Phi_{ttt}  (s)]}^2 +(1+s)^{4}\norm{ [V_{ttx}(s),\Phi_{tt} (s),\Phi_{ttx}(s)]}^2\right)ds\\
%		\leq &C\left(\norm{V_0}^2_3+\norm{M_0}^2_2+\norm{\Phi_0}^2_3+\delta_{0}\right).
%	\end{split}
%\end{equation}
\end{proposition}
Similar to the proof for the global existence,  with the help of the assumption $bp'(\bar{\rho})-a\mu\bar{\rho}>0$, the decay rates of $V_{x}$ and $\Phi$ can be obtained. The main ideas of the proof for Proposition \ref{decay rate} come from \cite{N1} and \cite{Nishihara2}, but additional efforts are needed to deal with the term $g$ in $\eqref{va1d1f}_{2}$. The time-weighted energy estimates \eqref{thm1.1c12} for $V_{x}$ and $\Phi$ follows from Lemmas \ref{sec.le4.1}-\ref{sec.le4.3} and a coarse decay rate for $V_{t}$ can also be obtained simultaneously. The refined decay rate of $V_{t}$ follows from Lemma \ref{sec.le4.4}. Here, due to the similarity, we only give the detailed proof of Lemmas \ref{sec.le4.1}-\ref{sec.le4.3} to see how we deal with the coupling of $V$ and $\Phi$.

\begin{lemma}\label{sec.le4.1}
Under the assumption of  Proposition \ref{1.1}, we have
\begin{equation*}
(1+t)\norm{[V_x,V_{t},\Phi,\Phi_{x}]}^2+\int_{0}^{t}(1+s)(\norm{V_t}^2+\norm{\Phi_t }^2)ds \leq C\left(\norm{V_0}^2_3+\norm{M_0}^2_2+\norm{\Phi_0}^2_4+\delta_{0}\right).
\end{equation*}
\end{lemma}
\begin{proof}
Let $k=0$ in \eqref{va1d2sumt}, we have
	\begin{equation}\label{dect}
		\begin{split}
			&\frac{d}{dt}\left\{\frac{1}{2}\|V_{t}\|^2+\frac{\mu D}{2a}\int_{\R} \bar{\rho}\Phi_{x}dx\right\}+\alpha\|V_{t}\|^2+\frac{\mu}{a}\int_{\R}\bar{\rho}\Phi_{t}^2dx \\[2mm]
			&+\frac{d}{dt}\left\{\frac{1}{2}\int_{\R} p'(\bar{\rho})V_{x}^2dx-\mu \int_{\R}\bar{\rho}\Phi V_{x}dx+\frac{\mu b}{2a}\int_{\R}\bar{\rho}\Phi^2dx\right\}\\[2mm]
			=:	&J_{1}^{(0)}+J_{2}^{(0)}+J_{3}^{(0)}+J_{4}^{(0)},
		\end{split}
	\end{equation}
	with
	\begin{equation*}
		\begin{split}
			J_{1}^{(0)}=&\frac{1}{2}\int_{\R} p''(\bar{\rho})\bar{\rho}_{t}V_{x}^2dx-\int_{\R} p''(\bar{\rho})\bar{\rho}_{x}V_{x} V_{t}dx -\mu \int_{\R}\bar{\rho}_{t}\Phi V_{x}dx\\
			&+\frac{\mu b}{2a}\int_{\R}\bar{\rho}_{t}\Phi^2dx+\frac{\mu D}{2a}\int_{\R} \bar{\rho}_{t}\Phi_{x}^2dx-\frac{\mu D}{a}\int_{\R} \bar{\rho}_{x}\Phi_{x} \Phi_{t}dx,
		\end{split}
	\end{equation*}
	\begin{equation*}
		\begin{split}
			J_{2}^{(0)}=\int_{\R} p''(\bar{\rho})\bar{\rho}_{x}V_{x}V_{t}dx
			-\int_{\R}\mu V_{x}\Phi_{x}  V_{t}dx-\frac{\mu}{b}\int_{\R}\bar{\rho}_{x}\left(aV_{x}-b\Phi\right) V_{t}dx,
		\end{split}
	\end{equation*}
	\begin{equation*}
		J_{3}^{(0)}=-\int_{\R}h_{x} V_{t}dx-\int_{\R}f_{x} V_{t}dx,\ \ \ \ \ J_{4}^{(0)}=\frac{\mu}{a}\int_{\R} \bar{\rho}\Phi_{t} gdx.
	\end{equation*}
	
Now we deal with the terms of the right hand side of \eqref{dect} one by one. It follows from Lemma \ref{prodecay} that
\begin{equation*}
J_{1}^{(0)}\leq C\delta_{0} (1+t)^{-1}\int_{\R}[V_x^2+\Phi^2+\Phi_{x}^2]dx+C\delta_{0}\int_{\R}[V_t^2+\Phi_{t}^2]dx,
\end{equation*}
where we have used that
\begin{equation*}
	\begin{split}
		&-\int_{\R} p''(\bar{\rho})\bar{\rho}_{x}V_{x} V_{t}dx-\frac{\mu D}{a}\int_{\R} \bar{\rho}_{x}\Phi_{x} \Phi_{t}dx\\
		\leq & \delta_{0}\int_{\R}[V^2_t+\Phi_{t}]^2 dx+\frac{C}{\delta_{0}}\norm{\bar{\rho}_x}_{L^{\infty}}^2\int_{\R}[V_x^2+\Phi_{x}^2]dx\\
		\leq & \delta_{0}\int_{\R}[V^2_t+\Phi_{t}]^2 dx+C\delta_{0} (1+t)^{-1}\int_{\R}[V_x^2+\Phi_{x}^2]dx.
	\end{split}
\end{equation*}
Similarly, we can estimate the first and the third term of $J_{2}^{(0)}$, and for the second term, by using  the {\it a priori} assumption \eqref{asstime1}, the Sobolev and Cauchy-Schwartz inequalities, we have
\begin{equation*}
\begin{split}
-\mu\int_{\R}V_x\Phi_xV_t dx\leq &\frac{\alpha}{4}\int_{\R}V^2_t dx+C\norm{\Phi_x}^2_{L^{\infty}}\int_{\R}V_x^2dx\\
\leq &\frac{\alpha}{4}\int_{\R}V^2_t dx+C\norm{\Phi_x}\norm{\Phi_{xx}}\int_{\R}V_x^2dx\\[2mm]
\leq &\frac{\alpha}{4}\int_{\R}V^2_t dx+C(1+t)^{-\frac{5}{2}}\int_{\R}V_x^2dx,
\end{split}
\end{equation*}
which implies that
\begin{equation*}
	J_{2}^{(0)}\leq C (1+t)^{-1}\int_{\R}[V_x^2+\Phi^2]dx+C\left(\frac{\alpha}{4}+\delta_{0}\right)\int_{\R}V_t^2dx.
\end{equation*}

Recalling \eqref{JF} and \eqref{JH} when $k=0$, we can estimate $J_{3}^{(0)}$ as
\begin{equation*}
\begin{split}
J_{3}^{(0)}
\le&-\frac{1}{2}\frac{d}{dt}\int_{\R}[p'(V_{x}+\bar{\rho})-p'(\bar{\rho})]V_{x}^2dx+\frac{1}{2}\frac{d}{dt}\int_{\R}\frac{(V_t+\frac{1}{\alpha}(q(\bar{\rho}))_x)^2}{(V_x+\bar{\rho})^2}V_{x}^2dx \\[2mm]
&+C\varepsilon_{0}\int_{\R}V_{t}^2dx+C(1+t)^{-1}\int_{\R}V_{x}^2dx+C\delta_{0}(1+t)^{-\frac{5}{2}},
\end{split}
\end{equation*}
where we have used \eqref{asstime1}, Lemma \ref{prodecay}, \eqref{va1d1f} and the following estimates:
\begin{equation*}
\left\|\left(\frac{(V_t+\frac{1}{\alpha}(q(\bar{\rho}))_x)^2}{(V_x+\bar{\rho})^2}\right)_{x}\right\|_{L^{\infty}}+\left\|\left(\frac{(V_t+\frac{1}{\alpha}(q(\bar{\rho}))_x)^2}{(V_x+\bar{\rho})^2}\right)_{t}\right\|_{L^{\infty}} \leq C(1+t)^{-1}.
\end{equation*}
Noticing that
\begin{equation*}
\norm{g}^2\le C\left(\norm{\bar{\phi}_t}^2+\norm{\bar{\phi}_{xx}}^2\right)\le C\delta_{0}^2(1+t)^{-\frac{3}{2}}
\end{equation*}
 is insufficient to warrant the decay rate $(1+t)^{-1}$ of $\norm{\Phi}^2$, we estimate $J_{4}^{(0)}$ as follows
\begin{equation*}
\begin{split}
J_{4}^{(0)}=\frac{\mu}{a}\int_{\R}\bar{\rho}\Phi_t g dx=&\frac{d}{dt}\left(\frac{\mu}{a}\int_{\R}\bar{\rho}\Phi g dx\right)-\frac{\mu}{a}\int_{\R}\bar{\rho}_{t}\Phi g dx-\frac{\mu}{a}\int_{\R}\bar{\rho}\Phi g_{t} dx \\[2mm]
\leq &\frac{d}{dt}\left(\frac{\mu}{a}\int_{\R}\bar{\rho}\Phi g dx\right)+C(1+t)^{-1}\int_{\R}\Phi ^2 dx+C(1+t)\int_{\R} \left[ g_{t}^2+\bar{\rho}_{t}^2g^{2}\right] dx\\[2mm]
\leq &\frac{d}{dt}\left(\frac{\mu}{a}\int_{\R}\bar{\rho}\Phi g dx\right)+C(1+t)^{-1}\int_{\R}\Phi ^2 dx+C\delta_{0}(1+t)^{-\frac{5}{2}}.
\end{split}
\end{equation*}
The estimation from $J_{1}^{(0)}$ to $J_{4}^{(0)}$ updates \eqref{dect} as
	\begin{equation}\label{va1d5dec1}
	\begin{split}
			&\frac{d}{dt}\left\{\frac{1}{2}\int_{\R} p'(\bar{\rho})V_{x}^2dx-\mu \int_{\R}\bar{\rho}\Phi V_{x}dx+\frac{\mu b}{2a}\int_{\R}\bar{\rho}\Phi^2dx\right\}\\[2mm]
		&+\frac{d}{dt}\left\{\frac{1}{2}\|V_{t}\|^2+\frac{\mu D}{2a}\int_{\R} \bar{\rho}\Phi_{x}dx\right\}+\frac{\alpha}{2}\|V_{t}\|^2+\frac{\mu}{2a}\int_{\R}\bar{\rho}\Phi_{t}^2dx \\[2mm]
	\leq &C(1+t)^{-1}\int_{\R}[V_x^2+\Phi^2+\Phi_{x}^2]dx+\frac{d}{dt}\left(\frac{\mu}{a}\int_{\R}\bar{\rho}\Phi g dx\right)+C\delta_{0}(1+t)^{-\frac{5}{2}}\\
	&-\frac{1}{2}\frac{d}{dt}\int_{\R}[p'(V_{x}+\bar{\rho})-p'(\bar{\rho})]V_{x}^2dx+\frac{1}{2}\frac{d}{dt}\int_{\R}\frac{(V_t+\frac{1}{\alpha}(q(\bar{\rho}))_x)^2}{(V_x+\bar{\rho})^2}V_{x}^2dx.
	\end{split}
\end{equation}
Multiplying  \eqref{va1d5dec1} with $(1+t)$, we have
\begin{equation}\label{dectr1}
\begin{split}
&\frac{d}{dt}\left((1+t)\int_{\R} \left[\frac{1}{2}p'(\bar{\rho})V_x^2-\mu\bar{\rho}\Phi V_x+\frac{\mu b}{2a}\bar{\rho}\Phi^2\right]dx\right)\\[2mm]
&+\frac{d}{dt}\left((1+t)\int_{\R}\left[ \frac{1}{2}V_t^2+\frac{D\mu }{2a}\bar{\rho}\Phi_x^2\right]dx\right)+\frac{1}{2}(1+t)\left(\int_{\R}\alpha V_t^2dx+\frac{\mu }{a}\int_{\R}\bar{\rho}\Phi_t ^2dx\right)\\[2mm]
=&\int_{\R} \left[\frac{1}{2}p'(\bar{\rho})V_x^2-\mu\bar{\rho}\Phi V_x+\frac{\mu b}{2a}\bar{\rho}\Phi^2\right]dx+\int_{\R}\left[ \frac{1}{2}V_t^2+\frac{D\mu }{2a}\bar{\rho}\Phi_x^2\right]dx\\[2mm]
&+C\int_{\R}\left[V_x^2+\Phi^{2}+\Phi_{x}^{2}\right]dx+\frac{d}{dt}\left((1+t)\frac{\mu}{a}\int_{\R}\bar{\rho}\Phi g dx\right)-\frac{\mu}{a}\int_{\R}\bar{\rho}\Phi g dx+C\delta_{0} (1+t)^{-\frac{3}{2}}\\[2mm]
&+\frac{1}{2}\frac{d}{dt}\left((1+t)\int_{\R}\frac{(V_t+\frac{1}{\alpha}(q(\bar{\rho}))_x)^2}{(V_x+\bar{\rho})^2}V_{x}^2dx\right)-\frac{1}{2}\int_{\R}\frac{(V_t+\frac{1}{\alpha}(q(\bar{\rho}))_x)^2}{(V_x+\bar{\rho})^2}V_{x}^2dx\\
&-\frac{1}{2}\frac{d}{dt}\left((1+t)\int_{\R}[p'(V_{x}+\bar{\rho})-p'(\bar{\rho})]V_{x}^2dx\right)+\frac{1}{2}\int_{\R}[p'(V_{x}+\bar{\rho})-p'(\bar{\rho})]V_{x}^2dx.
\end{split}
\end{equation}
 Integrating \eqref{dectr1} over $[0,t]$ and using Proposition \ref{mainpro} give us that
 \begin{equation*}
\begin{split}
(1+t)\norm{[V_x,V_{t},\Phi,\Phi_{x}]}^2+\int_{0}^{t}(1+s)(\norm{V_t}^2+\norm{\Phi_t }^2)dt \leq C\left(\norm{V_0}^2_3+\norm{M_0}^2_2+\norm{\Phi_0}^2_3+\delta_{0}\right),
 \end{split}
\end{equation*}
where we have used the following estimates
\begin{equation*}
 \begin{split}
 (1+t)\frac{\mu}{a}\int_{\R}\bar{\rho}\Phi g dx\leq &C\delta_{0}(1+t)\int_{\R}\Phi^{2}dx+\frac{C}{\delta_{0}} (1+t)\int_{\R}g^{2}dx\\
 \leq &\delta_{0}(1+t)\int_{\R}\Phi^{2}dx+C\delta_{0}
 \end{split}
 \end{equation*}
and
 \begin{equation*}
 \begin{split}
\frac{\mu}{a}\int_{0}^{t}\int_{\R}\bar{\rho}\Phi g dxdt\leq &C\int_{0}^{t}\int_{\R}\Phi^{2}dxdt+\int_{0}^{t}\int_{\R}g^{2}dxdt
 \leq C\left(\norm{V_0}^2_3+\norm{M_0}^2_2+\norm{\Phi_0}^2_3+\delta_{0}\right).
 \end{split}
 \end{equation*}
Thus, the proof of Lemma \ref{sec.le4.1} is completed.
\end{proof}

\begin{lemma}\label{sec.le4.2}
Under the assumption of Proposition \ref{1.1}, we have
\begin{equation}\label{con.4.2}
\begin{split}
&(1+t)^2\|[V_{xt},V_{xx},\Phi_{x},\Phi_{xx}]\|^2+\int_{0}^{t}[(1+s)^2\|[V_{xt},\Phi_{xt}]\|^2+(1+s)\|[V_{xx},\Phi_{x},\Phi_{xx}]\|^2]ds\\
\leq& C\left(\norm{V_0}^2_3+\norm{M_0}^2_2+\norm{\Phi_0}^2_3+\delta_{0}\right).
\end{split}
\end{equation}
\end{lemma}

\begin{proof} The proof consists of the following three steps. 	

	{\bf Step 1.} Letting $k=1$ in \eqref{va1d2sumt}, we have
\begin{equation}\label{vxtdec}
	\begin{split}
		&\frac{d}{dt}\left\{\frac{1}{2}\|V_{xt}\|^2+\frac{\mu D}{2a}\int_{\R} \bar{\rho}\Phi_{xx}^2dx\right\}+\alpha\|V_{xt}\|^2+\frac{\mu}{a}\int_{\R}\bar{\rho}\Phi_{xt}^2dx \\[2mm]
		&+\frac{d}{dt}\left\{\frac{1}{2}\int_{\R} p'(\bar{\rho})V_{xx}^2dx-\mu \int_{\R}\bar{\rho}\Phi_{x} V_{xx}dx+\frac{\mu b}{2a}\int_{\R}\bar{\rho}\Phi_{x}^2dx\right\}\\[2mm]
		&	=:J_{1}^{(1)}+J_{2}^{(1)}+J_{3}^{(1)}+J_{4}^{(1)}+J_{1,0},
	\end{split}
\end{equation}
with
\begin{equation*}
	\begin{split}
		J_{1}^{(1)}=&\frac{1}{2}\int_{\R} p''(\bar{\rho})\bar{\rho}_{t}V_{xx}^2dx-\int_{\R} p''(\bar{\rho})\bar{\rho}_{x}V_{xx} V_{xt}dx -\mu \int_{\R}\bar{\rho}_{t}\Phi_{x}V_{xx}dx\\
		&+\frac{\mu b}{2a}\int_{\R}\bar{\rho}_{t}\Phi_{x}^2dx+\frac{\mu D}{2a}\int_{\R} \bar{\rho}_{t}\Phi_{xx}^2dx-\frac{\mu D}{a}\int_{\R} \bar{\rho}_{x}\Phi_{xx}\Phi_{xt}dx,
	\end{split}
\end{equation*}
\begin{equation*}
	\begin{split}
		J_{2}^{(1)}=\int_{\R} \left(p''(\bar{\rho})\bar{\rho}_{x}V_{x}\right)_{x} V_{xt}dx
		-\int_{\R}\mu (V_{x}\Phi_{x})_{x} V_{xt}dx-\frac{\mu}{b}\int_{\R}\left[\bar{\rho}_{x}\left(aV_{x}-b\Phi\right)\right]_{x} V_{xt}dx,
	\end{split}
\end{equation*}
\begin{equation*}
	J_{3}^{(1)}=-\int_{\R}h_{xx} V_{xt}dx-\int_{\R}f_{xx} V_{xt}dx,\ \ \ \ \ J_{4}^{(1)}=\frac{\mu}{a}\int_{\R} \bar{\rho}\Phi_{xt} g_{x}dx
\end{equation*}
and
\begin{equation*}
	J_{1,0}=\int_{\R}p''(\bar{\rho})\bar{\rho}_{x}V_{xx} V_{xt}dx-\mu \int_{\R}\bar{\rho}_{xx}\Phi V_{xt}dx-\mu\int_{\R}\bar{\rho}_{x}\Phi_{x} V_{xt}dx.
\end{equation*}
	
Similar to Lemma \ref{sec.le4.1}, the term $J_{4}^{(1)}$ can be estimated as
 \begin{equation*}
\begin{split}
J_{4}^{(1)}=\frac{\mu }{a}\int_{\R}\bar{\rho}\Phi_{xt} g_xdx=&\frac{d}{dt}\left(\frac{\mu}{a}\int_{\R}\bar{\rho}\Phi_{x} g_{x} dx\right)-\frac{\mu}{a}\int_{\R}\bar{\rho}_{t}\Phi_{x} g_{x} dx-\frac{\mu}{a}\int_{\R}\bar{\rho}\Phi_{x} g_{xt} dx \\[2mm]
\leq &\frac{d}{dt}\left(\frac{\mu}{a}\int_{\R}\bar{\rho}\Phi_{x} g_{x} dx\right)+C(1+t)^{-1}\int_{\R}\Phi_{x}^2 dx+C(1+t)\int_{\R} \left[ g_{xt}^2+\bar{\rho}_{t}^2g_{x}^{2}\right] dx\\[2mm]
\leq &\frac{d}{dt}\left(\frac{\mu}{a}\int_{\R}\bar{\rho}\Phi_{x} g_{x}dx\right)+C(1+t)^{-1}\int_{\R}\Phi_{x}^2 dx+C\delta_{0}(1+t)^{-\frac{7}{2}}.
\end{split}
\end{equation*}
Recalling  \eqref{JH} and \eqref{JF}, we have
\begin{equation*}
\begin{split}
J_{3}^{(1)}=&-\int_{\R}h_{xx}V_{xt}dx-\int_{\R}f_{xx}V_{xt}dx\\
\le&-\frac{1}{2}\frac{d}{dt}\int_{\R}[p'(V_x+\bar{\rho})-p'(\bar{\rho})]V^2_{xx}dx+\frac{1}{2}\frac{d}{dt}\int_{\R}\frac{(V_t+\frac{1}{\alpha}(q(\bar{\rho}))_x)^2}{(V_x+\bar{\rho})^2}V_{xx}^2dx+C\varepsilon_{0}\int_{\R}V_{xt}^2dx \\[2mm]
&+C(1+t)^{-1}\int_{\R}V_{xx}^2dx+C(1+t)^{-2}\int_{\R} [V_{x}^2+V_{t}^2]dx+C\delta_{0}(1+t)^{-\frac{7}{2}}.
\end{split}
\end{equation*}
Substituting the above inequalities into \eqref{vxtdec} leads to
 \begin{equation}\label{vxtdec1}
\begin{split}
&\frac{1}{2}\frac{d}{dt}\int_{\R} V_{xt}^2dx+\frac{1}{2}\frac{d}{dt}\int_{\R}\left(p'(\bar{\rho})V_{xx}^2dx-2\mu\bar{\rho}\Phi_xV_{xx} +\frac{\mu b\bar{\rho}}{a}\Phi_x^2 \right)dx\\[2mm]
&+\frac{D\mu }{2a}\frac{d}{dt}\int_{\R}\bar{\rho}\Phi_{xx}^2dx+\frac{\mu }{a}\int_{\R}\bar{\rho}\Phi_{xt}^2dx+\alpha\int_{\R}V_{xt}^2dx\\[2mm]
\leq&\frac{1}{2}\int_{\R}\left(\frac{\mu \bar{\rho}}{a}\Phi_{xt}^2+\alpha V_{xt}^2 \right)dx+C(1+t)^{-1}\int_{\R}\left[V_{xx}^2+\Phi_{x}^{2}+\Phi_{xx}^{2}\right]dx\\[2mm]
&+ C\delta_{0}(1+t)^{-\frac{7}{2}}+C(1+t)^{-2}\int_{\R} [V_{x}^2+V_{t}^2+\Phi^2]dx+\frac{d}{dt}\left(\frac{\mu}{a}\int_{\R}\bar{\rho}\Phi_{x} g_{x} dx\right)\\
&+\frac{1}{2}\frac{d}{dt}\int_{\R}\frac{(V_t+\frac{1}{\alpha}(q(\bar{\rho}))_x)^2}{(V_x+\bar{\rho})^2}V_{xx}^2dx -\frac{1}{2}\frac{d}{dt}\int_{\R}[p'(V_x+\bar{\rho})-p'(\bar{\rho})]V^2_{xx}dx.
\end{split}
\end{equation}
{\bf Step 2.} Multiplying $\eqref{va1d1f}_{1}$ with $-V_{xx}$ and $\eqref{va1d1f}_{2}$ with $-\frac{\mu\bar{\rho}}{a}\Phi_{xx}$, integrating the result over $\R$, and using the integration by parts, we obtain
 \begin{equation*}
\begin{split}
&\frac{d}{dt}\int_{\R}\left( \frac{\alpha}{2}V_{x}^2+V_{xt}V_{x}\right)dx+\frac{\mu}{2a}\frac{d}{dt}\int_{\R}\bar{\rho}\Phi_{x}^2dx\\[2mm]
&+\int_{\R}\left(p'(\bar{\rho})V_{xx}^2-2\mu \bar{\rho}\Phi_{x}V_{xx}+\frac{b\mu\bar{\rho}}{a}\Phi_{x}^2\right)dx+\frac{\mu D}{a}\int_{\R}\bar{\rho} \Phi_{xx}^2dx\\[2mm]
=&\int_{\R}V_{xt}^2dx-\int_{\R}p''(\bar{\rho})\bar{\rho}_{x}V_{x}V_{xx}dx+\mu\int_{\R}V_{x}\Phi_{x}V_{xx}dx\\[2mm]
&+\mu\int_{\R}V_{x}\bar{\phi}_{x}V_{xx}dx+\int_{\R}(h_{x}+f_{x})V_{xx}dx+\mu\int_{\R}\bar{\rho}_{x}\Phi_{x}V_{x}dx\\[2mm]
&+\frac{\mu}{2a}\int_{\R}\bar{\rho}_{t}\Phi_{x}^2dx-\frac{\mu}{a}\int_{\R}\bar{\rho}_{x}\Phi_{t}\Phi_{x}dx-\frac{b\mu}{a}\int_{\R}\bar{\rho}_{x}\Phi_{x}\Phi dx-\int_{\R}\frac{\mu\bar{\rho}}{a}g\Phi_{xx}dx.
\end{split}
\end{equation*}
Integrating the last term by parts twice leads to
\begin{equation*}
\begin{split}
-\int_{\R}\frac{\mu\bar{\rho}}{a}g\Phi_{xx}dx=&-\frac{\mu}{a}\int_{\R}\left(\bar{\rho}g\right)_{xx}\Phi dx\\
\leq &(1+t)^{-1}\int_{\R}\Phi^2 dx +C(1+t)\int_{\R}\left( g_{xx}^2+\bar{\rho}_{xx}^2g^2+\bar{\rho}_{x}^2g_{x}^2\right)dx\\
\leq &(1+t)^{-1}\int_{\R}\Phi^2 dx +C\delta_{0}(1+t)^{-\frac{5}{2}}.
\end{split}
\end{equation*}
Thus we arrive at
\begin{equation}\label{t3c}
\begin{split}
&\frac{d}{dt}\int_{\R}\left( \frac{\alpha}{2}V_{x}^2+V_{xt}V_{x}\right)dx+\frac{\mu}{2a}\frac{d}{dt}\int_{\R}\bar{\rho}\Phi_{x}^2dx\\[2mm]
&+\int_{\R}\left[p'(\bar{\rho})V_{xx}^2-2\mu \bar{\rho}\Phi_{x}V_{xx}+\frac{b\mu\bar{\rho}}{a}\Phi_{x}^2\right]dx+\frac{\mu D}{a}\int_{\R}\bar{\rho} \Phi_{xx}^2dx\\[2mm]
\leq &2\int_{\R}V_{xt}^2dx+C(1+t)^{-1}\int_{\R}[V_{x}^{2}+\Phi^{2}+\Phi_{t}^2+V_{t}^2]dx\\[2mm]
&+C\delta_{0}(1+t)^{-\frac{5}{2}}+C\varepsilon_{0}\int_{\R}[V_{xx}^{2}+\Phi_{x}^{2}]dx.
\end{split}
\end{equation}
Here, we have used the expressions of $h_{x}$ and $f_{x}$ in \eqref{va1dh} and \eqref{va1d5d}.

{\bf Step 3.}  Multiplying $(1+t)$ to $\D[2\times\eqref{vxtdec1}+\frac{\alpha}{4}\times\eqref{t3c}]$ and integrating the result over  $[0,t]$, we get
\begin{equation}\label{le4.2.ste31}
\begin{split}
&(1+t)\|[V_{x},V_{xt},V_{xx},\Phi_{x},\Phi_{xx}]\|^2+\int_{0}^{t}(1+s)\|[V_{xt},\Phi_{xt},V_{xx},\Phi_{x},\Phi_{xx}]\|^2ds\\
\leq &C\left(\norm{V_0}^2_3+\norm{M_0}^2_2+\norm{\Phi_0}^2_3+\delta_{0}\right).
\end{split}
\end{equation}
Similarly, multiplying \eqref{vxtdec1} by $(1+t)^{2}$  and integrating the resulting equation  over $[0,t]$, we  have from \eqref{le4.2.ste31} that
\begin{equation}\label{le4.2.ste32}
(1+t)^2\|[V_{xt},V_{xx},\Phi_{x},\Phi_{xx}]\|^2+\int_{0}^{t}(1+s)^2\|[V_{xt},\Phi_{xt}]\|^2ds\leq C\left(\norm{V_0}^2_3+\norm{M_0}^2_2+\norm{\Phi_0}^2_3+\delta_{0}\right),
\end{equation}
where we have used the following estimates
\begin{equation*}
 \begin{split}
 (1+t)\frac{\mu}{a}\int_{\R}\bar{\rho}\Phi_{x}g_{x} dx\leq C\delta_{0}(1+t)\int_{\R}\Phi_{x}^{2}dx+\frac{C}{\delta_{0}} (1+t)\int_{\R}g_{x}^{2}dx
 \leq \delta_{0}(1+t)\int_{\R}\Phi_{x}^{2}dx+C\delta_{0},
 \end{split}
 \end{equation*}

 \begin{equation*}
 \begin{split}
\frac{\mu}{a}\int_{0}^t\int_{\R}\bar{\rho}\Phi_{x}g_{x} dxds\leq &C\int_{0}^t\int_{\R}\Phi_{x}^{2}dxds+C\int_{0}^t\int_{\R}g_{x}^{2}dx ds\leq  C\left(\norm{V_0}^2_3+\norm{M_0}^2_2+\norm{\Phi_0}^2_3+\delta_{0}\right),
 \end{split}
 \end{equation*}

 \begin{equation*}
 \begin{split}
 (1+t)^2\frac{\mu}{a}\int_{\R}\bar{\rho}\Phi_{x}g_{x} dx\leq &C\delta_{0}(1+t)^2\int_{\R}\Phi_{x}^{2}dx+\frac{C}{\delta_{0}} (1+t)^2\int_{\R}g_{x}^{2}dx
 \leq \delta_{0}(1+t)^2\int_{\R}\Phi_{x}^{2}dx+C\delta_{0}
 \end{split}
 \end{equation*}
and
 \begin{equation*}
 \begin{split}
\frac{\mu}{a}\int_{0}^t(1+s)\int_{\R}\bar{\rho}\Phi_{x}g_{x} dxds\leq& C\int_{0}^t(1+s)\int_{\R}\Phi_{x}^{2}dxds+C\int_{0}^t(1+s)\int_{\R}g_{x}^{2}dxds\\
 \leq &C\left(\norm{V_0}^2_3+\norm{M_0}^2_2+\norm{\Phi_0}^2_3+\delta_{0}\right).
 \end{split}
 \end{equation*}
Then the combination of \eqref{le4.2.ste31} and \eqref{le4.2.ste32} gives \eqref{con.4.2}. The proof of Lemma \ref{sec.le4.2} is completed.
\end{proof}

\begin{lemma} \label{sec.le4.3}
Under the assumption of  Proposition \ref{1.1}, we have
\begin{equation*}
\begin{split}
&(1+t)^3\|[V_{xxx},V_{xxt},\Phi_{xx},\Phi_{xxx}]\|^2
+\int_{0}^{t}(1+s)^3\|[V_{xxt},\Phi_{xxt}]\|^2+(1+t)^2\|[V_{xxx},\Phi_{xx},\Phi_{xxx}]\|^2ds\\
\leq &C\left(\norm{V_0}^2_3+\norm{M_0}^2_2+\norm{\Phi_0}^2_3+\delta_{0}\right).
\end{split}
\end{equation*}
\end{lemma}
\begin{proof}
The proof consists of three steps. 	

	{\bf Step 1.} Let $k=2$ in \eqref{va1d2sumt}, we have
	\begin{equation*}
		\begin{split}
			&\frac{1}{2}\frac{d}{dt}\int_{\R} V_{xxt}^2dx+\frac{1}{2}\frac{d}{dt}\int_{\R}\left[p'(\bar{\rho})V_{xxx}^2-2\mu\bar{\rho}\Phi_{xx}V_{xxx} +\frac{\mu b\bar{\rho}}{a}\Phi_{xx}^2\right]dx\\[2mm]
			&+\frac{D\mu }{2a}\frac{d}{dt}\int_{\R}\bar{\rho}\Phi_{xxx}^2dx+\alpha\int_{\R}V_{xxt}^2dx+\frac{\mu }{a}\int_{\R}\bar{\rho}\Phi_{xxt}^2dx\\[2mm]
				&	=:J_{1}^{(2)}+J_{2}^{(2)}+J_{3}^{(2)}+J_{4}^{(2)}+\sum_{\ell<2}J_{2,\ell},
		\end{split}
	\end{equation*}
with
\begin{equation*}
\begin{split}
J_{1}^{(2)}=&\frac{1}{2}\int_{\R} p''(\bar{\rho})\bar{\rho}_{t}V_{xxx}^2dx-\int_{\R} p''(\bar{\rho})\bar{\rho}_{x}V_{xxx}V_{txx}dx -\mu \int_{\R}\bar{\rho}_{t}\Phi_{xx} V_{xxx}dx\\
&+\frac{\mu b}{2a}\int_{\R}\bar{\rho}_{t}\Phi_{xx}^2dx+\frac{\mu D}{2a}\int_{\R} \bar{\rho}_{t}\Phi_{xxx}^2dx-\frac{\mu D}{a}\int_{\R} \bar{\rho}_{x}\Phi_{xxx} \Phi_{txx}dx,
\end{split}
\end{equation*}
\begin{equation*}
\begin{split}
J_{2}^{(2)}=\int_{\R}\left(p''(\bar{\rho})\bar{\rho}_{x}V_{x}\right)_{xx}V_{txx}dx
-\int_{\R}(V_{x}\Phi_{x})_{xx}V_{txx}dx-\frac{\mu}{b}\int_{\R} [\bar{\rho}_{x}\left(aV_{x}-b\Phi\right)]_{xx} V_{txx}dx,
\end{split}
\end{equation*}
\begin{equation*}
J_{3}^{(2)}=-\int_{\R}h_{xxx} V_{txx}dx-\int_{\R}f_{xxx} V_{txx}dx,\ \ \ \ \ J_{4}^{(2)}=\frac{\mu}{a}\int_{\R} \bar{\rho}\Phi_{txx} g_{xx}dx
\end{equation*}
and
\begin{equation*}
J_{2,\ell}=\int_{\R}\partial^{2-\ell}_{x}(p'(\bar{\rho}))\partial_{x}^{\ell}V_{xx} V_{txx}dx-\mu \int_{\R}\partial^{2-\ell}_{x}\bar{\rho}_{x}\partial^\ell_{x}\Phi V_{txx}dx-\mu\int_{\R}\partial^{2-\ell}_{x}\bar{\rho}\partial^\ell_{x}\Phi_{x} V_{txx}dx.
\end{equation*}
It is similar to estimate $J_{4}^{(2)}$ as in Lemma \ref{sec.le4.1} that
	\begin{equation*}
		\begin{split}
			\frac{\mu}{a}\int_{\R}\bar{\rho}\Phi_{xxt} g_{xx} dx=&\frac{d}{dt}\left(\frac{\mu}{a}\int_{\R}\bar{\rho}\Phi_{xx}g_{xx} dx\right)-\frac{\mu}{a}\int_{\R}\bar{\rho}_{t}\Phi_{xx}g_{xx}dx-\frac{\mu}{a}\int_{\R}\bar{\rho}\Phi_{xx}g_{xxt} dx \\[2mm]
			\leq &\frac{d}{dt}\left(\frac{\mu}{a}\int_{\R}\bar{\rho}\Phi_{xx} g_{xx} dx\right)+C(1+t)^{-1}\int_{\R}\Phi_{xx}^2 dx+C(1+t)\int_{\R} \left[ g_{xxt}^2+\bar{\rho}_{t}^2g_{xx}^{2}\right] dx\\[2mm]
			\leq &\frac{d}{dt}\left(\frac{\mu}{a}\int_{\R}\bar{\rho}\Phi_{xx} g_{xx} dx\right)+C(1+t)^{-1}\int_{\R}\Phi_{xx}^2 dx+C\delta_{0}(1+t)^{-\frac{9}{2}}.
		\end{split}
	\end{equation*}
	Then, by a direct calculation, we have
	\begin{equation}\label{t7c}
		\begin{split}
			&\frac{1}{2}\frac{d}{dt}\int_{\R} V_{xxt}^2dx+\frac{1}{2}\frac{d}{dt}\int_{\R}\left[p'(\bar{\rho})V_{xxx}^2-2\mu\bar{\rho}\Phi_{xx}V_{xxx} +\frac{\mu b\bar{\rho}}{a}\Phi_{xx}^2\right]dx\\[2mm]
			&+\frac{D\mu }{2a}\frac{d}{dt}\int_{\R}\bar{\rho}\Phi_{xxx}^2dx+\frac{\mu }{a}\int_{\R}\bar{\rho}\Phi_{xxt}^2dx+\alpha\int_{\R}V_{xxt}^2dx\\[2mm]
			\leq & \frac{1}{2}\int_{\R}\left[\alpha V_{xxt}^2+\frac{\mu \bar{\rho}}{a}\Phi_{xxt}^2\right]dx+C(1+t)^{-1}\int_{\R}[V_{xxx}^2+\Phi_{xx}^{2}+\Phi_{xxx}^2+V_{xt}^2]dx\\
			&+C(1+t)^{-2}\int_{\R}[V_{xx}^2+\Phi_{x}^{2}]dx+C(1+t)^{-3}\int_{\R}[V_{x}^2+V_{t}^2+\Phi^2]dx\\
			&+C\delta_{0}(1+t)^{-\frac{9}{2}}+\frac{d}{dt}\left(\frac{\mu}{a}\int_{\R}\bar{\rho}\Phi_{xx} g_{xx} dx\right)\\
			&+\frac{1}{2}\frac{d}{dt}\int_{\R}\frac{(V_t+\frac{1}{\alpha}(q(\bar{\rho}))_x)^2}{(V_x+\bar{\rho})^2}V_{xxx}^2dx-\frac{1}{2}\frac{d}{dt}\int_{\R}[p'(V_{x}+\bar{\rho})-p'(\bar{\rho})]V_{xxx}^{2}dx.
		\end{split}
	\end{equation}

{\bf Step 2.}  Differentiating \eqref{va1d1f} with respect to $x$, we obtain
\begin{equation}\label{va1d1fx}
	\begin{cases}
		\begin{split}
			&V_{xtt}-(p'(\bar{\rho})V_x)_{xx}+\alpha V_{xt}+\mu(\bar{\rho}\Phi_x)_{x}=-(\mu V_x\Phi_x+\mu V_x\bar{\phi}_x)_{x}-h_{xx}-f_{xx},\\[2mm]
			&\Phi_{xt}-D\Phi_{xxx}-aV_{xx}+b\Phi_x=g_x.
		\end{split}
	\end{cases}
\end{equation}

Multiplying  $\eqref{va1d1fx}_{1}$ with $-V_{xxx}$ and multiplying $\eqref{va1d1fx}_{2}$ with  $-\frac{\mu\bar{\rho}}{a}\Phi_{xxx}$, we have

\begin{equation}\label{t6}
\begin{split}
&\frac{d}{dt}\int_{\R}\left( \frac{\alpha}{2}V_{xx}^2+V_{xxt}V_{xx}\right)dx+\frac{\mu}{2a}\frac{d}{dt}\int_{\R}\bar{\rho}\Phi_{xx}^2dx\\[2mm]
&+\int_{\R}\left[p'(\bar{\rho})V_{xxx}^2-2\mu \bar{\rho}\Phi_{xx}V_{xxx}+\frac{b\mu\bar{\rho}}{a}\Phi_{xx}^2\right]dx+\frac{\mu D}{a}\int_{\R}\bar{\rho} \Phi_{xxx}^2dx\\[2mm]
=&\int_{\R}V_{xxt}^2dx-2\int_{\R}p''(\bar{\rho})\bar{\rho}_{x}V_{xx}V_{xxx}dx-\int_{\R}(p'(\bar{\rho}))_{xx}V_{x}V_{xxx}dx\\[2mm]
&+\mu\int_{\R}(V_{x}\Phi_{x})_{x}V_{xxx}dx+\mu\int_{\R}(V_{x}\bar{\phi}_{x})_{x}V_{xxx}dx+\mu\int_{\R}\bar{\rho}_{x}\Phi_{x}V_{xxx}dx\\[2mm]
&+\int_{\R}(h+f)_{xx}V_{xxx}dx+\frac{\mu}{2a}\int_{\R}\bar{\rho}_{t}\Phi_{xx}^2dx-\frac{\mu}{a}\int_{\R}\bar{\rho}_{x}\Phi_{xt}\Phi_{xx}dx\\[2mm]
&+\mu\int_{\R}\bar{\rho}_{x}\Phi_{xx}V_{xx}dx-\frac{b\mu}{a}\int_{\R}\bar{\rho}_{x}\Phi_{x}\Phi_{xx}dx-\frac{\mu}{a}\int_{\R}\bar{\rho}g_{x}\Phi_{xxx}dx.
\end{split}
\end{equation}
Integrating the last term  by parts twice, and then we have
\begin{equation*}
\begin{split}
-\frac{\mu}{a}\int_{\R}\bar{\rho}g_{x}\Phi_{xxx}dx=&-\frac{\mu}{a}\int_{\R}\bar{\rho}g_{xxx}\Phi_{x} dx-\frac{\mu}{a}\int_{\R}\bar{\rho}_{xx}g_{x}\Phi_{x} dx-\frac{2\mu}{a}\int_{\R}\bar{\rho}_{x}g_{xx}\Phi_{x} dx \\
\leq &C\delta_{0}(1+t)^{-1}\int_{\R}\Phi_{x}^2 dx +\frac{C}{\delta_{0}}(1+t)\int_{\R}\left( g_{xxx}^2+\bar{\rho}^2_{xx}g^2_{x}+\bar{\rho}^2_{x}g^2_{xx}\right)dx\\
\leq &C\delta_{0}(1+t)^{-1}\int_{\R}\Phi_{x}^2 dx +C\delta_{0}(1+t)^{-\frac{7}{2}},
\end{split}
\end{equation*}
which along with \eqref{t6} yields
\begin{equation*}
\begin{split}
&\frac{d}{dt}\int_{\R}\left( \frac{\alpha}{2}V_{xx}^2+V_{xxt}V_{xx}\right)dx+\frac{\mu}{2a}\frac{d}{dt}\int_{\R}\bar{\rho}\Phi_{xx}^2dx\\[2mm]
&+\int_{\R}\left[p'(\bar{\rho})V_{xxx}^2-2\mu \bar{\rho}\Phi_{xx}V_{xxx}+\frac{b\mu\bar{\rho}}{a}\Phi_{xx}^2\right]dx+\frac{\mu D}{a}\int_{\R}\bar{\rho} \Phi_{xxx}^2dx\\[2mm]
\leq &2\int_{\R}V_{xxt}^2dx+C\varepsilon_{0}\int_{\R}[V_{xxx}^{2}+\Phi_{xx}^2]dx+C\delta_{0}(1+t)^{-\frac{7}{2}}\\
 &+C(1+t)^{-1}\int_{\R}[V_{xx}^{2}+\Phi_{x}^2+\Phi_{xt}^2+V_{xt}^2]dx+C(1+t)^{-2}\int_{\R}[V_{x}^{2}+V_{t}^2]dx.
\end{split}
\end{equation*}

{\bf Step 3.} Multiplying $(1+t)^j$ $(j=1,2)$ to $[\D2\times\eqref{t6}+\frac{\alpha}{4}\times\eqref{t7c}]$ and integrate the result over  $[0,t]$, we get by Lemma \ref{sec.le4.1} and Lemma \ref{sec.le4.2} that
\begin{equation}\label{le4.3.ste31}
\begin{split}
&(1+t)^2\|[V_{xx},V_{xxt},V_{xxx},\Phi_{xx},\Phi_{xxx}]\|^2+\int_{0}^{t}(1+s)^2\|[V_{xxx},V_{xxt},\Phi_{xx},\Phi_{xxx},\Phi_{xxt}]\|^2ds\\
\leq &C\left(\norm{V_0}^2_3+\norm{M_0}^2_2+\norm{\Phi_0}^2_3+\delta_{0}\right).
\end{split}
\end{equation}
Furthermore integrating \eqref{t7c} multiplied by $(1+t)^{j+1}$ $(j=1,2)$ over $[0,t]$, we obtain from \eqref{le4.3.ste31} that
\begin{equation*}
\begin{split}
& (1+t)^3\|[V_{xxt},V_{xxx},\Phi_{xx},\Phi_{xxx}]\|^2+\int_{0}^{t}(1+s)^3\|[V_{xxt},\Phi_{xxt}]\|^2ds\\
\leq & C\left(\norm{V_0}^2_3+\norm{M_0}^2_2+\norm{\Phi_0}^2_3+\delta_{0}\right).
\end{split}
\end{equation*}
Here, for $j=1,2$, we have used the following estimates
\begin{equation*}
	\begin{split}
		(1+t)^{j}\frac{\mu}{a}\int_{\R}\bar{\rho}\Phi_{xx}g_{xx} dx\leq &C\delta_{0}(1+t)^{j}\int_{\R}\Phi_{xx}^{2}dx+\frac{C}{\delta_{0}} (1+t)^j\int_{\R}g_{xx}^{2}dx\\
		\leq & \delta_{0}(1+t)^{j}\int_{\R}\Phi_{x}^{2}dx+C\delta_{0}(1+t)^{-\frac{5}{2}+j}\\
		\leq & \delta_{0}(1+t)^{j}\int_{\R}\Phi_{x}^{2}dx+C\delta_{0},
	\end{split}
\end{equation*}

\begin{equation*}
	\begin{split}
		\frac{\mu}{a}\int_{0}^t(1+s)^{j-1}\int_{\R}\bar{\rho}\Phi_{xx}g_{xx} dxds\leq &C\int_{0}^t(1+s)^{j-1}\int_{\R}\Phi_{xx}^{2}dx+C\int_{0}^t(1+s)^{j-1}\int_{\R}g_{xx}^{2}dxds\\
		\leq &C\int_{0}^t(1+s)^{j-1}\int_{\R}\Phi_{xx}^{2}dx+C\int_{0}^t(1+s)^{-\frac{9}{2}+j}ds\\
		\leq &C\left(\norm{V_0}^2_3+\norm{M_0}^2_2+\norm{\Phi_0}^2_3+\delta_{0}\right),
	\end{split}
\end{equation*}

\begin{equation*}
	\begin{split}
		(1+t)^{j+1}\frac{\mu}{a}\int_{\R}\bar{\rho}\Phi_{xx}g_{xx} dx\leq &C\delta_{0}(1+t)^{j+1}\int_{\R}\Phi_{xx}^{2}dx+\frac{C}{\delta_{0}} (1+t)^{j+1}\int_{\R}g_{xx}^{2}dx\\
		\leq & \delta_{0}(1+t)^{j+1}\int_{\R}\Phi_{x}^{2}dx+C\delta_{0}(1+t)^{-\frac{5}{2}+j}\\
		\leq & \delta_{0}(1+t)^{j+1}\int_{\R}\Phi_{x}^{2}dx+C\delta_{0}
	\end{split}
\end{equation*}
and
\begin{equation*}
	\begin{split}
		\frac{\mu}{a}\int_{0}^t(1+s)^{j}\int_{\R}\bar{\rho}\Phi_{xx}g_{xx} dxds\leq &C\int_{0}^t(1+s)^{j}\int_{\R}\Phi_{xx}^{2}dx+C\int_{0}^t(1+s)^{j}\int_{\R}g_{xx}^{2}dxds\\
		\leq &C\int_{0}^t(1+s)^{j}\int_{\R}\Phi_{xx}^{2}dx+C\int_{0}^t(1+s)^{-\frac{7}{2}+j}ds\\
		\leq &C\left(\norm{V_0}^2_3+\norm{M_0}^2_2+\norm{\Phi_0}^2_3+\delta_{0}\right).
	\end{split}
\end{equation*}
Thus, we complete the proof of Lemma \ref{sec.le4.3}.
\end{proof}

Similar to Lemma \ref{sec.le4.3}, we can get the following higher-order estimates for which we only outline the procedures without details for brevity.
\begin{lemma} \label{sec.le4.4}
	Under the assumption of   Proposition \ref{1.1}, we have
	\begin{equation}\label{refine1}
		\begin{split}
			(1+t)^3&\|[V_{tt},V_{xt},\Phi_{t},\Phi_{xt}]\|^2+(1+t)^2\|V_{t}\|^2\\[2mm]
			&+\int_{0}^{t}(1+s)^3\|[V_{tt},\Phi_{tt}]\|^2+(1+t)^2\|[V_{xt},\Phi_{xt},\Phi_{t}]\|^2ds\\
			\leq &C\left(\norm{V_0}^2_3+\norm{M_0}^2_2+\norm{\Phi_0}^2_3+\delta_{0}\right)
		\end{split}
	\end{equation}
	and
	
	\begin{equation}\label{refine2}
		\begin{split}
			(1+t)^4&\|[V_{xtt},V_{xxt},\Phi_{xt},\Phi_{xxt}]\|^2+(1+t)^3\|V_{xt}\|^2\\[2mm]
			&+\int_{0}^{t}(1+s)^4\|[V_{xtt},\Phi_{xtt}]\|^2ds+(1+t)^3\|[V_{xxt},\Phi_{xt},\Phi_{xxt}]\|^2ds\\
			\leq &C\left(\norm{V_0}^2_3+\norm{M_0}^2_2+\norm{\Phi_0}^2_4+\delta_{0}\right).
		\end{split}
	\end{equation}
	 \end{lemma}
\begin{proof}
	In fact, taking  integration over $\R\times (0,t)$ of equations $(1+t)^{i}\times \{2 [\partial_{t}\eqref{va1d1fxk}_{1}\times V_{tt}+\partial_{t}\eqref{va1d1fxk}_{2}\times \frac{\mu \bar{\rho}}{a}\Phi_{tt}]+\frac{\alpha}{4}[\partial_{t}\eqref{va1d1fxk}_{1}\times V_{t}+\partial_{t}\eqref{va1d1fxk}_{2}\times \frac{\mu \bar{\rho}}{a}\Phi_{t}]$\}, and  $(1+t)^{i+1}\times [\partial_{t}\eqref{va1d1fxk}_{1}\times V_{tt}+\partial_{t}\eqref{va1d1fxk}_{2}\times \frac{\mu \bar{\rho}}{a}\Phi_{tt}]$ for $i=0,1,2$, we get \eqref{refine1}. Taking  integration over $\R\times (0,t)$ of equations $(1+t)^{j}\times \{2 [\partial_{xt}\eqref{va1d1fxk}_{1}\times V_{xtt}+\partial_{xt}\eqref{va1d1fxk}_{2}\times \frac{\mu \bar{\rho}}{a}\Phi_{xtt}]+\frac{\alpha}{4}[\partial_{xt}\eqref{va1d1fxk}_{1}\times V_{xt}+\partial_{t}\eqref{va1d1fxk}_{2}\times \frac{\mu \bar{\rho}}{a}\Phi_{xt}]\}$, and $(1+t)^{j+1}\times [\partial_{xt}\eqref{va1d1fxk}_{1}\times V_{xtt}+\partial_{xt}\eqref{va1d1fxk}_{2}\times \frac{\mu \bar{\rho}}{a}\Phi_{xtt}]$ for $j=0,1,2,3$, we get \eqref{refine2}.
	\end{proof}

Note Proposition \ref{decay rate} is a direct consequence of Lemmas \ref{sec.le4.1}-\ref{sec.le4.4} shown
above. Thus, we can close the {\it a priori} assumptions \eqref{asstime1} by taking  $\epsilon$ to be sufficiently small in Proposition \ref{1.1}.\\
%Furthermore, we can obtain the decay rate of $V_{ttt}$ and $\Phi_{tt}$ by  taking  integration over $\R\times (0,t)$ to	the equations $(1+t)^{i}\times \{2 [\partial_{tt}\eqref{va1d1fxk}_{1}\times V_{ttt}+\partial_{tt}\eqref{va1d1fxk}_{2}\times \frac{\mu \bar{\rho}}{a}\Phi_{ttt}]+\frac{\alpha}{4}[\partial_{tt}\eqref{va1d1fxk}_{1}\times V_{tt}+\partial_{tt}\eqref{va1d1fxk}_{2}\times \frac{\mu \bar{\rho}}{a}\Phi_{tt}]$, and  $(1+t)^{i+1}\times [\partial_{tt}\eqref{va1d1fxk}_{1}\times V_{ttt}+\partial_{tt}\eqref{va1d1fxk}_{2}\times \frac{\mu \bar{\rho}}{a}\Phi_{ttt}]$, for $i=0,1,2,3,4$, ad follows.
%
%\begin{lemma}\label{sec.le4.6}
%Under the assumption of Proposition \ref{1.1}, we have
%\begin{equation}\label{con.4.6}
% \begin{split}
% &(1+t)^{5}\norm{[V_{ttt}(t),V_{ttx}(t),\Phi_{tt}(t),\Phi_{xtt}(t)]}^2 +(1+t)^4 \norm{V_{tt}(t)}^2\\[2mm]
% &+\int_{0}^t\left((1+s)^{5}\norm{ [V_{ttt}(s),\Phi_{ttt}  (s)]}^2 +(1+s)^{4}\norm{ [V_{ttx}(s),\Phi_{tt} (s),\Phi_{ttx}(s)]}^2\right)ds\\
% \leq &C\left(\norm{V_0}^2_3+\norm{M_0}^2_2+\norm{\Phi_0}^2_3+\delta_{0}\right).
% \end{split}
% \end{equation}
%
%\end{lemma}

Finally we are in a position to prove  Proposition \ref{1.1} and Theorem \ref{1-1}.\\

\noindent {\bf Proof of Proposition \ref{1.1}}. The first part of Proposition \ref{1.1} (global existence) is a consequence of Proposition \ref{local} and Proposition \ref{mainpro}. For the decay rate, we  have from Proposition \ref{decay rate} directly. \\

\noindent {\bf Proof of Theorem \ref{1-1}}. Notice that $M(x,t)=-V_t(x,t)$ and the transformation \eqref{221}, we get \eqref{infinitydecay} from \eqref{decayn} by the Sobolev inequality $\|f\|_{L^\infty(\R)}^2 \leq 2 \|f\|_{L^2(\R)}\|f_x\|_{L^2(\R)}$ and hence complete the proof.

\bigbreak \noindent \textbf{Acknowledgement}.
Q.Q. Liu was supported by the National Natural Science
Foundation of China (No. 12071153), Guangdong Basic and Applied Basic Research Foundation (No. 2021A1515012360) and the Fundamental Research Funds for the Central Universities (No. 2020ZYGXZR032). H.Y. Peng support from the National Natural Science Foundation
of China No. 11901115 and Natural Science Foundation of Guangdong Province (No.2019A1515010706). Z.A. Wang was supported in part by
the Hong Kong RGC GRF grant No. PolyU 15304720.

%\addcontentsline{toc}{section}{\\References}

\end{document}